\documentclass[10pt]{article} 
\usepackage[preprint]{tmlr}

\usepackage{hyperref}
\usepackage{url}
\usepackage{graphicx} 
\usepackage{natbib}

\usepackage{tabularx} 
\usepackage{booktabs} 

\usepackage{amsmath,amsthm,amssymb,mathtools}
\usepackage{enumitem}
\usepackage{xcolor}

\newcommand{\R}{\mathbb{R}}
\newcommand{\RR}{\mathbb{R}} 
\DeclareMathOperator{\Projr}{Proj_r}
\DeclareMathOperator*{\argmin}{argmin}
\DeclareMathOperator*{\trace}{trace}

\DeclareMathOperator*{\rank}{rank}
\def\EE{\mathbb{E}}
\renewcommand{\vec}[1]{\boldsymbol{#1}} 
\newcommand{\x}{\vec{x}}
\newcommand{\z}{\vec{z}}
\newcommand{\y}{\vec{y}}

\newcommand{\cS}{\mathcal{Z}} 
\newcommand{\cZ}{\mathcal{Z}}

\newcommand{\euc}{{}}
\newcommand{\fro}{{}} 
\newcommand{\ZZ}{z} 

\newtheorem{theorem}{Theorem}

\newtheorem{lemma}{Lemma}

\newtheorem{definition}{Definition}
\newtheorem{assumption}{Assumption}

\theoremstyle{remark}

\newtheorem{remark}{Remark}

\usepackage{newfloat}
\usepackage{listings}

\usepackage{booktabs}

\title{Gradient-Free Optimization for Matrix functions}

\author{\name Sawyer Allen \email sawyer\_allen@mines.edu \\
      \addr Unaffiliated
      \tmlrAND
      \name Cash Cherry \email ccherry@ucsb.edu \\
      \addr Department of Mathematics, UCSB
      \tmlrAND
      \name Aidan Eck \email aeck1@mines.edu \\
      \addr Unaffiliated
      \tmlrAND
      \name Stephen Becker \email stephen.becker@colorado.edu \\
      \addr Department of Applied Mathematics, CU Boulder
      \tmlrAND
      \name Daniel McKenzie \email dmckenzie@mines.edu \\
      \addr Department of Applied Mathematics and Statistics, Colorado School of Mines
      }
\let\tmlrAND\AND           
\let\AND\relax 
\usepackage{algorithm}
\usepackage{algorithmic}

\begin{document}

\maketitle

\begin{abstract}

We consider the task of optimizing smooth, possibly non-convex functions of a matrix variable given access only to directional derivatives rather than full gradients. This setting arises when fine-tuning large neural networks on consumer-grade hardware: the network's weights are matrices, memory constraints rule out backward-mode automatic differentiation, but directional derivatives remain available through forward mode.

We frame gradient estimation in this setting as a structured recovery problem, in the spirit of signal processing. From this perspective we provide three contributions. First, we introduce an alternative to the standard random gradient estimator; the difference corresponds to replacing the adjoint of the sampling operator with its pseudoinverse. Second, when the gradient satisfies an approximate low-rank condition, techniques from matrix sensing yield a family of highly accurate gradient estimators that drop into any first-order method. Third, we note that while the computational cost of such estimators is high, this can be amortized by combining them with a matrix-aware optimizer such as spectral descent. Specifically, the gradient estimator computes a factorization of the gradient, allowing for the projection step of spectral descent to be done at no extra cost.

We demonstrate our findings with two careful numerical experiments on synthetic functions with approximately low-rank gradients. We show that by exploiting this low-rank property one obtains much faster convergence to good approximate solutions.

\end{abstract}


\section{Introduction}
\label{sec:introduction}
In this paper we consider the following optimization problem
\begin{equation}
    \min_{X \in \R^{m\times n}} f(X)
    \label{eq:minimization_problem}
\end{equation}
where $f: \R^{m \times n} \to \R$ is Lipschitz differentiable. That is, we are interested in settings where the input to $f$ can naturally be thought of as a matrix. We shall call such functions {\em matrix functions}. Matrix functions arise in problems such as sensor localization \cite{tasissa2018exact}, quantum state tomography \cite{gross2010quantum}, and the training/fine-tuning of neural networks. While matrix functions may be reduced to the more familiar form $\tilde{f}: \mathbb{R}^N \to \mathbb{R}$ by defining $N = mn$ and simply stacking the columns of $X \in \mathbb{R}^{m\times n}$ to form a vector $\vec{x} \in \mathbb{R}^N$, retaining the matrix structure of $X$ allows for specialized optimization algorithms that may outperform generic, vectorized ones.

We study \eqref{eq:minimization_problem} in the gradient-free optimization (GFO) setting, where the full-gradient $\nabla f(X)$ is inaccessible but directional derivatives,
\begin{equation*}
    D_Zf(X):= \lim_{h\to 0}\frac{f(X + hZ) - f(X)}{h},
\end{equation*}
are readily available. Such settings frequently arise when finetuning large language models (LLMs) due to the asymmetry in memory cost between evaluating the model (a forward pass or inference) and computing a gradient via reverse-mode automatic differentiation (a backward pass). The former is possible on consumer-grade hardware, while the latter frequently is not \cite{zhang2022opt}. On the other hand, a directional derivative can be evaluated using forward-mode automatic differentiation and has a memory cost similar to that of a forward pass.

In the closely related setting of derivative-free optimization (DFO), even directional derivatives are inaccessible \cite{conn2009introduction}. However, directional derivatives may be approximated using finite differences:
\begin{equation}
   \frac{f(X+hZ) - f(X)}{h} \approx D_{Z}f(X), 
\end{equation}
for sufficiently small $h > 0$. Consequently GFO methods may be extended to this setting, see Appendix~\ref{ap:connection_to_dfo}. DFO---and to a lesser extent, GFO---for matrix functions has recently received much attention \cite{zhang2024revisiting,chenenhancing2025,lang2026powering,petrov2025leveraging}, largely applied to the problem of LLM finetuning. This paper explores the topic from a theoretical, problem-agnostic perspective so as to expand the design space for such algorithms.

Previous literature has explored gradient-free methods for functions of a vector input \cite{nesterov2017random,baydin2022gradients,belouze2022optimization,flugel2025beyond,kozak2021stochastic}. A common bottleneck to applying gradient-free methods to large problems is that, for generic objective functions $f$, they typically require $O(N)$ directional derivative evaluations to find an acceptable solution to \eqref{eq:minimization_problem}. Consequently, many works in GFO and DFO have proposed various notions of {\em low intrinsic dimensionality}, such as gradient sparsity, that can be exploited to design algorithms capable of finding an approximate solution using a sublinear-in-$N$ number of queries \cite{balasubramanian2018zeroth,wang2018stochastic,cai2021zeroth,cai2022zeroth,balasubramanian2022zeroth,qiugradient,grapiglia2025fully,yue2023zeroth,cartis2022dimensionality,li2023stochastic,malladi2023fine}. 

We frame the problem of estimating $\nabla f(X)$ as a matrix recovery problem. This leads us to a notion of low intrinsic dimensionality adapted to matrix functions: approximately low-rank gradients. We provide evidence that this property occurs in applications of interest, specifically LLM finetuning. Drawing upon existing tools from signal processing, we provide several new GFO algorithms---some of which exploit low-rank gradients, and some which are insensitive to this. We demonstrate the performance of such approaches in a controlled experiment using two toy problems. We conclude with some thoughts on how to extend these approaches to large-scale problems.   The accompanying code is available at \href{https://github.com/mines-opt-ml/gradient_free_matrix_opt}{https://github.com/mines-opt-ml/gradient\_free\_matrix\_opt}.

\begin{figure*}
    \centering
    \includegraphics[width=0.32\linewidth]{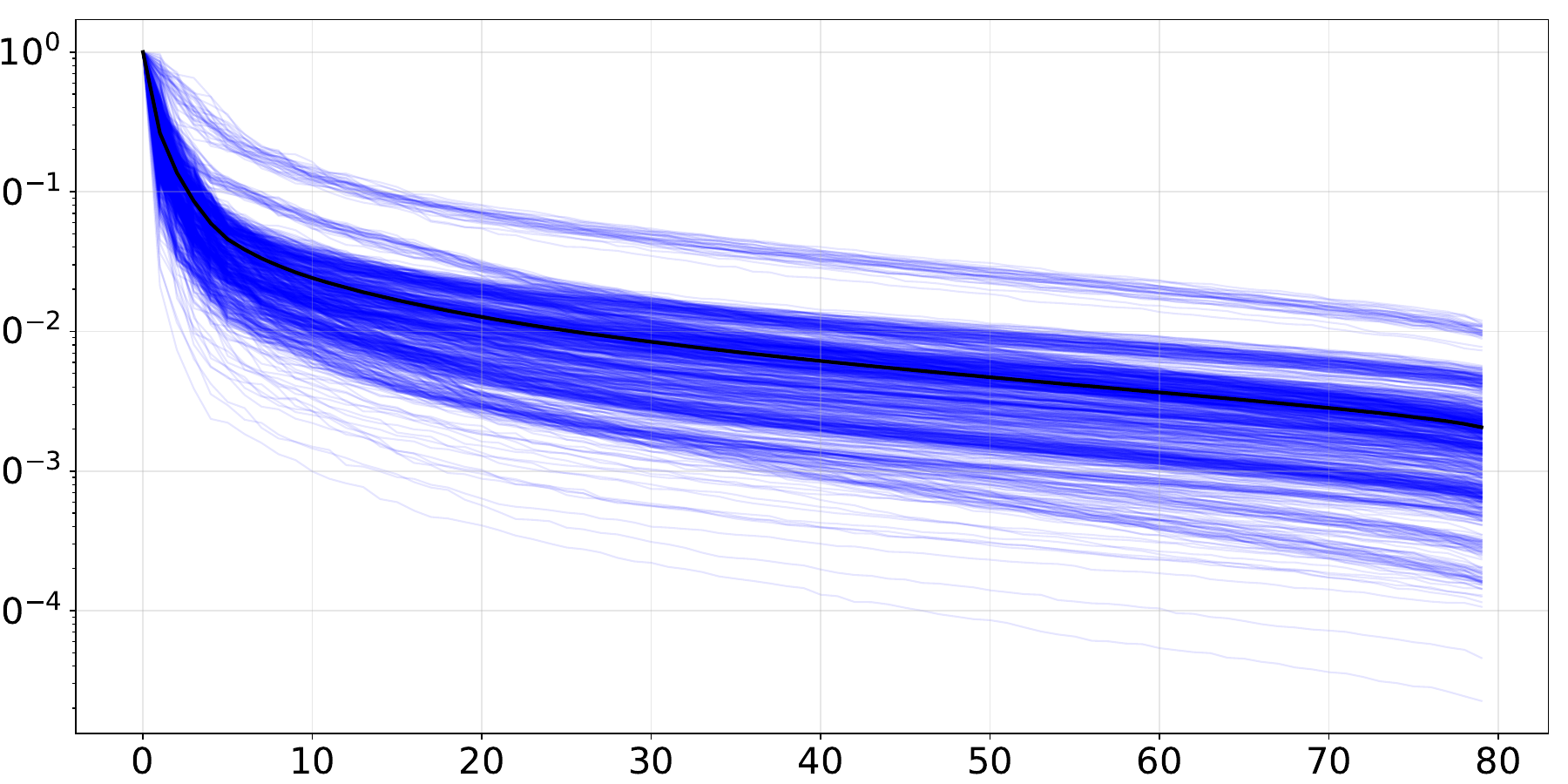} \hfill
    \includegraphics[width=0.32\linewidth]{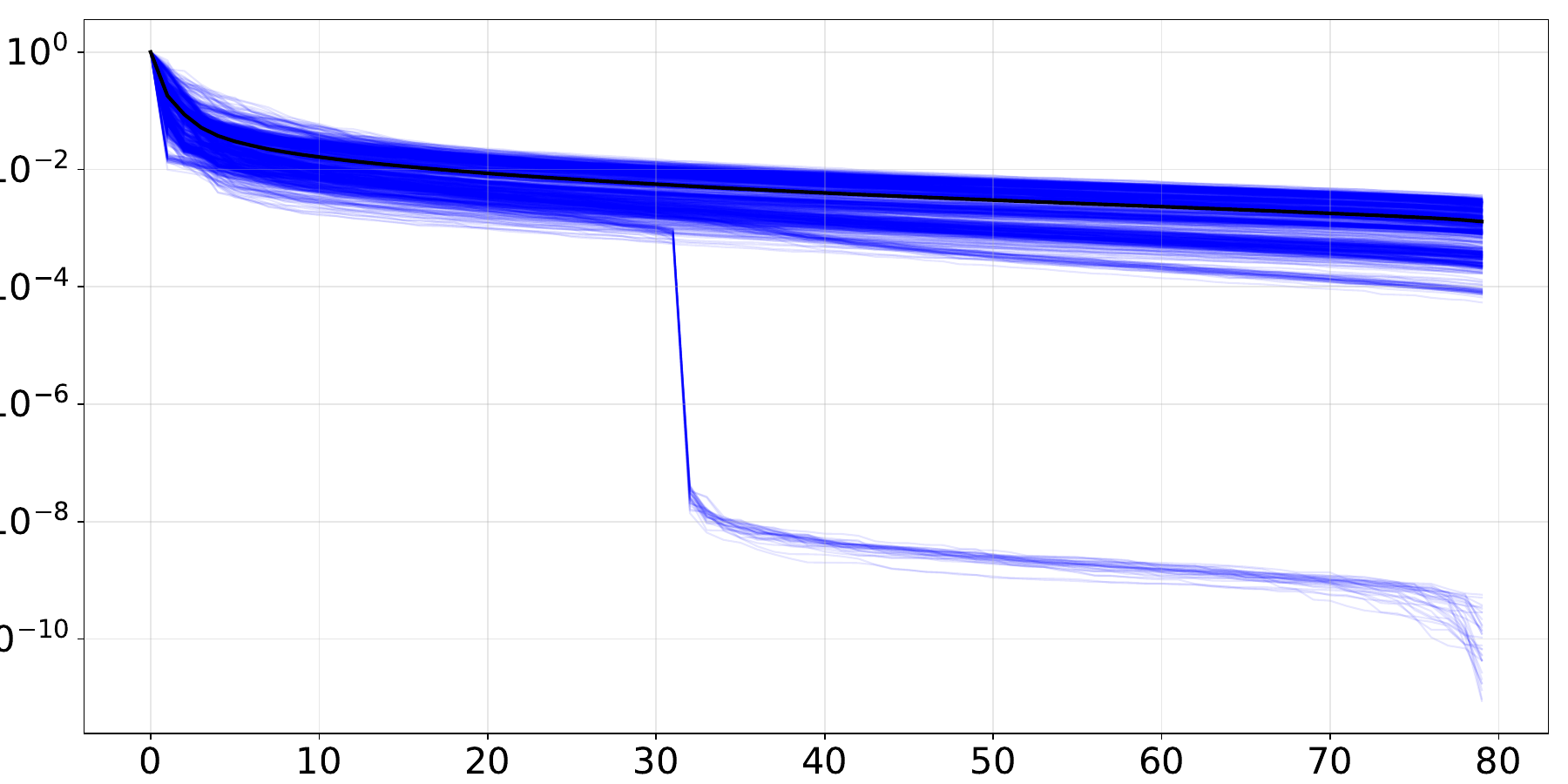} \hfill
    \includegraphics[width=0.32\linewidth]{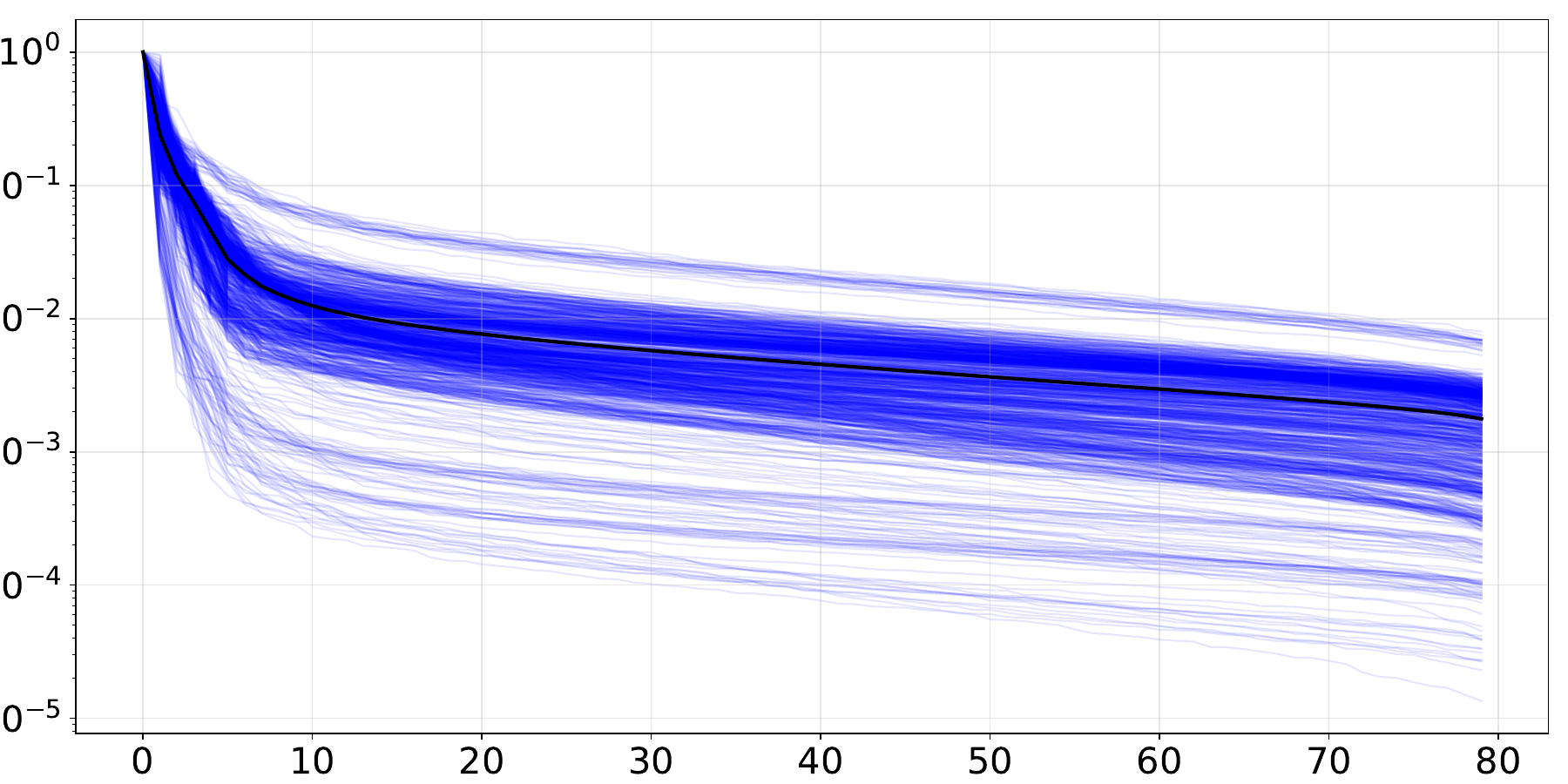}
    \caption{Singular values of gradients with respect to $Q,K$, and $V$ (left, center, and right respectively) matrices across all attention heads of {\tt OPT-2.7B}. We use a mini-batch of size $32$ drawn from the SST-2 data set. Each blue line represents the singular values of a gradient matrix for a particular attention head, normalized by dividing through by the largest singular value. The black line represents the mean. Note that each gradient matrix is of size $2560\times 80$. }
    \label{fig:opt_singular_values}
\end{figure*}

\subsection{Notation and Conventions}
We  consider $m \times n$ matrices. Without loss of generality, we take $m \leq n$. We shall refer to the (reduced) singular value decomposition of any matrix $A \in \mathbb{R}^{m\times n}$ as $A = U_A\Sigma_AV_A^{\top}$ where $U_A \in \mathbb{R}^{m\times m}$ and $V_A\in \mathbb{R}^{n\times m}$ have orthonormal columns and $\Sigma_A \in \mathbb{R}^{{m}\times {m}}$ is diagonal with the singular values $\sigma_1(A),\ldots, \sigma_{m}(A)$ of $A$ along the diagonal.

When we wish to emphasize that the optimization variable is a matrix we write in capital letters like $X \in \R^{m\times n}$. 
Occasionally we write $\vec{x} = \mathrm{vec}(X) \in \mathbb{R}^{N}$ to denote the vector obtained by stacking the columns of $X,$ where $N = mn$. For $A,B \in \mathbb{R}^{m\times n}$ we denote the Frobenius inner product and its induced norm as
\begin{equation}\label{eq:innerprod}
    \langle A,B\rangle_\fro = \mathrm{trace}(A^{\top}B), \quad \quad \|A\|_\fro = \sqrt{\langle A,A\rangle_\fro}. 
\end{equation}
This defines the same geometry over $\R^{m\times n}$ as the Euclidean geometry over $\R^N$. 
We additionally use $\|\vec{x}\|$ to denote the Euclidean norm of a vector, $\|A\|_2 = \sigma_1(A)$ to denote the spectral norm of a matrix, and $\|A\|_{\star} = \sum_{i=1}^{m} \sigma_i(A)$ to denote the nuclear norm of a matrix (which is the dual norm to the spectral norm). 

For any $A \in \R^{m\times n}$ we define $\Projr(A)$ to be the nearest matrix of rank at most $r$ to $A$:
\begin{equation}
    \Projr(A) = \argmin_{B: \rank(B) \leq r}\|A - B\|_\fro
\end{equation}
$\Projr(A)$ may be computed in closed form using the SVD:
\begin{equation}
    \Projr(A) = U_A\tilde{\Sigma}_AV_{A}^{\top}.
    \label{eq:compute_r_proj}
\end{equation}
where $\tilde{\Sigma}_A = \mathrm{diag}(\sigma_1(A),\ldots, \sigma_r(A),0, \ldots, 0)$. At a minimum, we assume the following:

\begin{assumption}
    We assume that $f$ is bounded below. That is, there exists a value $f_{\mathrm{low}}\in \mathbb{R}$ such that 
    \begin{equation}
        f(X) \geq f_{\mathrm{low}} \qquad \forall X \in \mathbb{R}^{m\times n}.
    \end{equation}
    \label{assumption:bounded_below}
\end{assumption}

\begin{assumption}
 We assume $f:\R^{m\times n} \to \R $ is Lipschitz differentiable, or $L$-smooth, 
 \begin{enumerate}[label=(\alph*), ref=\theassumption(\alph*)]
    \item\label{assumption:L_smooth} with respect to the Frobenius inner product:
    $$
        \| \nabla f ( X ) - \nabla f ( Y ) \|_\fro \leq L \| X - Y \|_\fro.
    $$
    \item\label{assumption:L_smooth_nuclear} with respect to the spectral/nuclear norm pairing:
     $$
         \| \nabla f ( X ) - \nabla f ( Y ) \|_\star \leq L_\star \| X - Y \|_2.
     $$
 \end{enumerate}
\end{assumption}

By the equivalence of norms on finite-dimensional vector spaces Assumptions~\ref{assumption:L_smooth} and \ref{assumption:L_smooth_nuclear} are equivalent, albeit with different constants.

\begin{remark}
    To simplify the exposition, we focus on functions of a single matrix variable. However, our techniques may easily be extended to functions of a set of matrix variables: $f(\mathbf{X})$ where $\mathbf{X} = \{X_1,\ldots, X_{\ell}\}$ and $X_{i}\in\R^{m_i\times n_i}$. Such functions are ubiquitous in LLM pre-training and finetuning.
    \label{remark:sets_matrices}
\end{remark}


\section{Related Work}\label{sec:PriorWork}

\paragraph{Random gradient method} The most common method for gradient-free optimization (GFO) was studied by \citet{nesterov2017random} under the name {\em random gradient method}, which iterates as
\begin{equation}
    \vec{x}_{k+1} = \vec{x}_k - \alpha_kD_{\vec{z}}f(\vec{x}_k)\vec{z}.
\end{equation}
In the context of machine learning, this method is also referred to as the {\em forward gradient method} \citep{baydin2022gradients,belouze2022optimization,renscaling}. Variants which sample multiple $\vec{z}_i$ per iteration, i.e. $\vec{x}_{k+1} = \vec{x}_k - \alpha \vec{g}_k$ where
\begin{equation}
    \vec{g}_k:= \frac{1}{d}\sum_{i=1}^d D_{\vec{z}_i}f(\vec{x}_k)\vec{z}_i
    \label{eq:random_gradient_method}
\end{equation}
are a natural, although somewhat under-studied extension \citep{malladi2023fine,flugel2025beyond}. \citet{kozak2021stochastic} propose a gradient-free method which uses orthogonalized directions $\vec{z}_i$ to obtain the projection of $\nabla f(\vec{x}_k)$ to random subspace.

\paragraph{Derivative-free optimization (DFO)} DFO considers the setting where only function evaluations are available. Assuming noise-free function evaluations, directional derivatives may be approximated to arbitrary accuracy as
 \begin{equation}
     D_zf(\vec{x}) \approx \frac{f(\vec{x} + h\vec{z}) - f(\vec{x})}{h}
     \label{eq:connection_to_dfo}
 \end{equation}
by selecting a sufficiently small $h$ (see Lemma~\ref{lemma:directional_derivatives}). 
The derivative-free analogue of \eqref{eq:random_gradient_method}, i.e.
\begin{equation}
     \vec{g}_k:= \frac{1}{d}\sum_{i=1}^d \frac{f(\x_k + h\z_i) - f(\x_k)}{h}\z_i, 
     \label{eq:RGE}
\end{equation}
is known as the Random Gradient Estimator (RGE). DFO methods which use the RGE as a drop-in replacement for the gradient in first-order methods are widely used in practice \cite{liu2020primer, zhang2024revisiting}, and well-studied in theory \citep{nesterov2017random,ghadimi2013stochastic,berahas2022theoretical}. This is certainly not the only approach to derivative-free optimization, and we mention also model-based~\cite{powell2006newuoa} algorithms; line-search based methods~\cite{karmanov1974convergence,golovin2019gradientless,bergou2020stochastic,gorbunov2019stochastic,kim2025curvature}; direct/pattern search methods~\cite{nelder1965simplex,hough2001asynchronous,audet2006mesh}; and evolutionary algorithms~\cite{hansen2001completely,loshchilov2014computationally,salimans2017evolution}.

\paragraph{Matrix optimization} In the DFO setting, LOZO \cite{chenenhancing2025} attempts to exploit low-rank structure in the gradients of matrix functions by replacing the random sampling directions\footnote{Note they use centered---instead of forward---finite differences.} in the RGE \eqref{eq:RGE} with random low-rank matrices:
\begin{equation}
     G_k:= \frac{1}{d}\sum_{i=1}^d \frac{f(X_k + hU_iV_i^{\top}) - f(X_k)}{h}U_iV_i^{\top}.
     \label{eq:LRGE}
\end{equation}
This approach is extended by \citet{fang2026lozo+}. In the first-order setting \citet{martens2015optimizing,gupta2018shampoo,jordan2024muon} and many others have analyzed matrix optimizers such as K-FAC, Shampoo, and Muon. These optimizers have been shown to outperform `generic' vector optimizers such as Adam~\cite{kingma2014adam} for certain tasks; notably LLM pre-training \citep{liu2025muon}. Here, we focus on spectral descent \citep{carlson2015preconditioned,carlson2015stochastic}, in which $\nabla f(X_k)$ is computed and then projected onto the set of semi-orthogonal matrices:
\begin{equation}
    \mathcal{O}_{m,n} = \{O \in \mathbb{R}^{m\times n}: O^{\top}O = I_{n,n}\},
    \label{eq:semi-orthogonal-matrices}
\end{equation}
whence the algorithm iterates as 
\begin{equation}
    X_{k+1} = X_k - \alpha_k \mathrm{Proj}_{\mathcal{O}_{m,n}}(\nabla f(X_k)).
    \label{eq:spectral_descent}
\end{equation}
Analysis of spectral descent \citep{carlson2015preconditioned,bernstein2024old,shulginbeyond} typically proceeds by interpreting this projection as the steepest descent direction with respect to the spectral norm. 

Several recent works have adapted Muon/spectral descent to the DFO/GFO setting by combining a gradient estimator of the form \eqref{eq:RGE}, \eqref{eq:LRGE}, or \eqref{eq:random_gradient_method} with projection to $\mathcal{O}_{m,n}$ \citep{lang2026powering,petrov2025leveraging} However, these works treat gradient estimation and projection as independent sub-routines. We show that by reframing gradient estimation one can achieve both at the same time.

\section{Low-rank gradients}
\label{sec:low_rank_gradients}
We say $A \in \mathbb{R}^{m\times n}$ is $(\delta, r)$-approximately-low-rank if it is well-approximated by a low-rank matrix:
   \begin{equation}
        \|A - \Projr(A)\|_\fro \leq \delta\|A\|_\fro.
        \label{eq:low_rank_frobenius}
    \end{equation}
We say $f: \mathbb{R}^{m\times n} \to \mathbb{R}$ has $(\delta, r)$-approximately-low-rank gradients if, for all $X \in \mathbb{R}^{m\times n}$, the gradient $\nabla f(X)$ is  $(\delta, r)$-approximately-low-rank.

\begin{assumption} $f$ has $(\delta, r)$-approximately-low-rank gradients for some $r\ll m$ and $\delta \ll 1$.
\label{assumption:low-rank}
\end{assumption}

Low-rank structure in the weight matrices of deep neural networks has long been acknowledged, and is the basis for the popular LoRA 
fine-tuning framework \cite{hu2021lora}. Several recent works~\cite{zhao2024galore,chenenhancing2025,lang2026powering} have observed that the gradients associated to these weight matrices are also low-rank. For example, \citet{chenenhancing2025} track the gradients associated to the $Q$ and $V$ matrices of a particular attention head while fine-tuning the {\tt OPT-1.3B} \cite{zhang2022opt} LLM. They show that the singular values of these gradients decay rapidly, indicating approximate low-rank structure.

We verify this with a related but larger model ({\tt OPT-2.7B}) and a different fine-tuning task (SST-2 \cite{socher-etal-2013-recursive}). We train for just a single step of stochastic gradient descent (batch size of 32), but examine the structure of $\frac{\partial f}{\partial Q}, \frac{\partial f}{\partial K}$, and $\frac{\partial f}{\partial V}$ for all 1024 attention heads. The results are striking. Across all 1024 heads, we find that $59\%$ of the $K$ matrices, $83\%$ of the $Q$ matrices, and $58\%$ of the $V$ matrices are $(\delta = 0.05, r = 4)$-approximately low-rank (see Assumption~\ref{assumption:low-rank}). A similar experiment with {\tt gpt-2} \cite{radford2019language} revealed even more extreme low rank structure: of the 144 attention heads, 124 of the $K$ matrices, 123 of the $Q$ matrices, and $111$ of the $V$ matrices are $(\delta = 0.05, r = 2)$-approximately low-rank.

\section{Gradient Recovery}
\label{sec:Recovery_Algorithms}
The random gradient method \eqref{eq:random_gradient_method} uses one approach for determining a step direction from directional derivatives. We frame this---in the matrix function context---as a problem of estimating the full gradient given only directional derivatives. This allows us to introduce novel tools from the signal processing into gradient-free optimization.

Suppose we have selected $d < N$ sampling directions $Z_1,\ldots, Z_d \in \mathbb{R}^{m\times n}$. Consider the induced {\bf sensing operator}
\begin{subequations}
\label{eq:signal_matrix_recovery}
\begin{align}
    & \mathcal{Z}: \mathbb{R}^{m\times n} \to \mathbb{R}^d \\
    & \left[\mathcal{Z}(A)\right]_i = \langle Z_i, A\rangle_\fro.
\end{align}
\end{subequations}
An extensive line of work has studied the generic inverse problem of approximating $A$ given $\mathcal{Z}(A)$. We call this the {\em matrix recovery problem}. Because 
\begin{equation*}
    y_i := D_{Z_i}f(X) = \langle Z_i,\nabla f(X)\rangle
\end{equation*}
we may write $\y := \mathcal{Z}( \nabla f(X) ) \in \mathbb{R}^d$ and define the {\bf gradient recovery problem as}
\begin{equation}
    \textit{Approximate } \nabla f(X) \textit{ given } (\mathcal{Z},\y),
    \label{eq:gradient_estimation_problem}
\end{equation}
which we recognize as a special case of matrix recovery. We review several representative approaches and adapt them to \eqref{eq:gradient_estimation_problem}. We test these approaches empirically in Section~\ref{sec:numerical_experiments} and analyze a subset of them theoretically in Section~\ref{sec:proposed_algorithm}.

\subsection{Rank-agnostic Recovery algorithms}
\label{sec:rank-agnostic_Recovery_Algorithms}

The simplest algorithm for solving \eqref{eq:gradient_estimation_problem} is to apply the adjoint:
\begin{equation*}
    G_k = \mathcal{Z}^{*}(\y) = \sum_{i=1}^d y_i Z_i.
\end{equation*}
With $y_i$ as above, we note that {\em the adjoint sensing operator recovers the gradient estimator used in the random gradient method} \eqref{eq:random_gradient_method}, up to a normalizing factor of $1/d$. We absorb this factor into the sampling distribution, see Section~\ref{sec:sampling_distribution}. We refer to this approach as {\tt adjoint}. 

An improvement on {\tt adjoint} can be made by applying the pseudoinverse operator $\mathcal{Z}^\dagger$ to $\y$ instead of the adjoint. In the case that the \emph{Gram matrix} $A \in \R^{d \times d}$ with entries $A_{ij} = \langle Z_i, Z_j \rangle_\fro$ is invertible, this corresponds to 
\begin{equation}
G_k = \mathcal{Z}^\dagger ( \y ) = \sum_{i=1}^d w_i Z_i, \quad \quad \vec{w} = A^{-1} \y,
\label{eq:pseudo_inverse}
\end{equation}
see \citet{luenberger} for a proof. We call this approach {\tt pseudoinverse}. This is closely related to the simplex gradient in DFO \cite{custodio2007using,regis2015calculus}. 

\subsection{Rank-aware Recovery Algorithms}
When $f$ satisfies Assumption~\ref{assumption:low-rank} we may refine \eqref{eq:gradient_estimation_problem} to incorporate this information. Specifically, we may phrase gradient recovery as a constrained optimization problem:
\begin{equation}
    G_k = \argmin_{G \in \mathbb{R}^{m\times n}} \|\mathcal{Z}(G)  - \y \|_\euc \text{\it subject to } \mathrm{rank}(G) \leq r. \label{eq:lr_gradient_estimation_problem}
\end{equation}

Similar formulations exist for vector gradients, see \cite{choromanski2020provably,wang2018stochastic,cai2022zeroth}. Although \eqref{eq:lr_gradient_estimation_problem} is a non-convex problem, one can still apply first-order methods such as projected gradient descent. The resulting iteration, with step sizes $\eta^t > 0$,
\begin{equation}
    G^{t+1} = \Projr\left(G^t - \eta^t \mathcal{Z}^{\star}(\mathcal{Z}(G^t) - \y ) \right),
    \label{eq:iterative_hard_thresholding}
\end{equation}
is guaranteed to converge linearly under certain conditions on $\mathcal{Z}$ \cite{jain2010guaranteed,davenport2016overview}. We select $\eta^t$ using a rule proposed by \citet{tanner2013normalized}, and, following \citet{davenport2016overview}, we call this algorithm iterative hard thresholding, or {\tt IHT}.

Because the optimization variable $G$ is assumed low-rank, it can be written as $G = UV$ where $U \in \mathbb{R}^{m\times r}$ and $V \in \mathbb{R}^{r \times n},$ leading to a  class of algorithms which avoids the repeated SVD computations required to compute $\Projr$. While many algorithms in this class exist, we focus on two prototypical methods: alternating minimization {\tt AltMin} \cite{jain2012lowrankmatrixcompletionusing, hardt2014understanding} and Burer-Monteiro gradient descent, or {\tt BM-GD}. Both methods reparameterize \eqref{eq:lr_gradient_estimation_problem} to
\begin{equation}\label{eq:reparamterized_objective}
(U_k,V_k) = 
\argmin_{(U, V)} 
\| \cZ (U V^\top) - \y \|.
\end{equation}
{\tt AltMin} freezes one matrix variable at a time and solves a standard least squares problem in the other: 
\begin{align*}
U^{t+1} &= \argmin_{U} \| \cZ ( U V^{t,\top} ) - \y \|, \\
V^{t+1} &= \argmin_{V} \| \cZ ( U^{t+1} V^\top ) - \y \|.
\end{align*}
We perform $T$ iterations of this procedure. Another approach is to minimize the objective \eqref{eq:reparamterized_objective} with gradient descent; a naive implementation suffers from a non-uniqueness problem, as $UV^\top = (U C^{-1})(V C)^\top$ for invertible $C \in \RR^{r \times r}.$ Algorithm 2 of \cite{procrustes}  employs a regularizer of the form $\| U U^\top - V V^\top \|$ to improve scaling of this problem. This is the algorithm we call {\tt BM-GD}, which is gradient descent applied to the objective
$$
(U_k,V_k) = 
\argmin_{(U, V)}  \underbrace{\| \cZ (U V^\top) - \y \|^2 + 
\frac{1}{8} \|
UU^\top - VV^\top
\|^2}_{:= g(U,V)}.
$$
We use a backtracking line search to set the step size $\eta^t$ and perform $T$ steps.


\subsection{On the selection of sampling directions}
\label{sec:sampling_distribution}
Analyzing both rank-agnostic and rank-aware recovery algorithms requires that $\mathcal{Z}$ is suitably well conditioned. For example, in the rank-aware case, we require $\mathcal{Z}$ to satisfy the {\em restricted isometry property}: 
\begin{definition}
\label{def:rip}
    We say $\cZ$ satisfies the $s$ restricted isometry property (RIP) if there exists a constant $\delta \in (0,1)$ if, for all $G \in \mathbb{R}^{m\times n}$ of rank at most $s$ we have that
    \begin{equation}
        (1-\delta)\|G\|_F^2 \leq \|\cZ(G)\|^2_2 \leq (1+\delta)\|G\|_F^2
    \end{equation}
    We call the smallest such $\delta$ the {\em restricted isometry constant} of $\cZ$, denoted $\delta_s(\cZ)$. 
\end{definition}
In signal processing the physics of the problem often constrain the construction of $\mathcal{Z}$, making it hard to guarantee this conditioning. However, in our setting there are no such constraints, and consequently we choose a sampling distribution for $\mathcal{Z}$---for all recovery algorithms---for which it is easy to prove the required conditioning holds. 

\begin{assumption} Each entry of each $Z_i$ is sampled independently from the Gaussian distribution $\mathcal{N}(0, 1/d)$.
\label{assumption:sampling_directions}
\end{assumption}

\section{Proposed Algorithms}
\label{sec:proposed_algorithm}
We propose two approaches for problem~\eqref{eq:minimization_problem}. The first is a simple meta-algorithm: we update the random gradient method by replacing the gradient estimator \eqref{eq:random_gradient_method} with the estimators of Section~\ref{sec:Recovery_Algorithms}, see Algorithm~\ref{alg:RA-GD}. We call this approach {\tt RA-GD}, where {\tt RA} is a placeholder for the recovery algorithm used. Note the subtle difference in how rank-aware and rank-agnostic recovery algorithms are treated. For rank-aware methods $Z_1,\ldots, Z_d$ only needs to be sampled once. This is because the success or failure of these approaches depends only on whether $\cZ$ satisfies the RIP (see Definition~\ref{def:rip}), which holds with high probability for a single draw of $Z_1,\ldots, Z_d$. The second approach drops the {\tt IHT} gradient estimator into the spectral descent algorithm \eqref{eq:spectral_descent} in a way that amortizes the cost of gradient estimation over the orthogonalization step.

In GFO, it is the number of directional derivative queries, not the number of iterations, that ultimately counts. Consequently, we analyze the convergence of all methods with respect to the number of queries required to find an $\varepsilon$-stationary point, that is, a point $X_k$ with
\begin{equation}
    \|\nabla f(X_k)\|_\fro^2 \leq \varepsilon.
    \label{eq:eps_stationarity}
\end{equation}

\begin{algorithm}
\caption{{\tt RA-GD} where {\tt RA} denotes the recovery algorithm used.}
\label{alg:RA-GD}
\begin{algorithmic}[1]
\STATE \textbf{Input:} $X_0$, choice of recovery algorithm {\tt RA}, sampling parameter $d$

\IF{{\tt RA} $\in \{\text{\tt IHT, AltMin, BM-GD} \}$ } 
    \STATE  Sample $Z_1,\ldots, Z_d$ satisfying Assumption~\ref{assumption:sampling_directions}.
\ENDIF
\FOR{ $k=0,\ldots, K-1$}
        \IF{{\tt RA} $\in\{{\tt adjoint, pseudoinverse}\}$}
            \STATE Sample $Z_1,\ldots, Z_d$ satisfying Assumption~\ref{assumption:sampling_directions}.
        \ENDIF
        \STATE Acquire $y_i = D_{Z_i}f(X_k)$ for $i=1,\ldots, d$. 
        \STATE Acquire $G_k = \text{\tt RA}(\vec{y}; Z_1,\ldots, Z_d)\approx \nabla f(X_k)$.
        \STATE Acquire step-size $\alpha_k$
    \STATE Iterate: $X_{k+1} = X_k - \alpha_k G_k$
\ENDFOR
\STATE \textbf{Return:} $X_K$.
\end{algorithmic}
\end{algorithm}

\subsection{The rank agnostic case}
\label{sec:rank_agnostic_convergence}

First, we analyze the $\text{\tt RA} = {\tt pseudoinverse}$ and ${\tt RA} = {\tt adjoint}$ cases. Neither of these algorithms exploit low-rank structure in $\nabla f(X_k)$. Consequently, both algorithms may be analyzed by simply vectorizing the problem. The results in this section either recover or modestly extend results in the existing literature. 

\begin{lemma}\label{lemma:moment_bounds_for_pseudoinverse}
    Suppose Assumptions \ref{assumption:L_smooth} and \ref{assumption:sampling_directions} hold. Then $G_k$ produced by {\tt pseudoinverse} satisfies 
    \begin{equation}
        G_k = \cS^{\dagger}(\y) = \mathrm{Proj}_{\mathrm{Span}(Z_1,\ldots, Z_d)}(\nabla f(X_k)).
    \end{equation}
    Furthermore,
    \begin{align*}
    \EE [ G_k | X_k ] = \frac {\min (d, mn)}{mn} \nabla f ( X_k ) \quad \text{ and } \quad \EE [ \| G_k \|_\fro^2 | X_k ] =  \frac{\min (d, mn)}{mn} \| \nabla f ( X_k ) \|_\fro^2.
    \end{align*}
\end{lemma}

Next we consider ${\tt RA} = {\tt adjoint}$. 
\begin{lemma}
    Suppose Assumptions~\ref{assumption:L_smooth} and \ref{assumption:sampling_directions} hold. Then $G_k$ produced by {\tt adjoint} satisfies
    \begin{align}
        & \mathbb{E}[G_k | X_k] = \nabla f(X_k) \quad \text{ and } \quad  \mathbb{E}\left[\left\|G_k\right\|_\fro^2 | X_k\right] = \left( \frac{d + mn+1}{d}\right)\|\nabla f(X_k)\|_\fro^2.
    \end{align}
    \label{lemma:moment_bounds_for_adjoint}
\end{lemma}
\begin{remark}
    Lemma~\ref{lemma:moment_bounds_for_adjoint} with $d=1$ is directly comparable to \citep[Theorem 3]{nesterov2017random} where it is shown (converting to our notation) that
    \begin{equation}
        \mathbb{E}\left[\left\|G_k\right\|_\fro^2\right] \leq (mn+4)\|\nabla f(X_k)\|_\fro^2.
    \end{equation}
    We improve upon this bound by reducing $4$ to $2$. Similar results can be found in \citet[Property 5]{belouze2022optimization}, \citet[Lemma 2]{malladi2023fine} (for the case where the $Z_i$ are sampled uniformly from a sphere), and \citet[Lemma 2.4]{berahas2022theoretical} (for the true zeroth-order case).
\end{remark}

Combining Lemmas~\ref{lemma:moment_bounds_for_pseudoinverse} and \ref{lemma:moment_bounds_for_adjoint} with a standard convergence-in-expectation lemma we obtain the convergence results shown in the first two rows of Table~\ref{tab:rank_agnostic_convergence}; see Theorem~\ref{thm:rankagnostic} in Appendix~\ref{appendix:rank-agnostic} for the proof. These results do not require Assumption~\ref{assumption:low-rank}. The takeaway is that while both have the same order of dependence on $K$, the dimensional dependence for {\tt pseudoinverse} ($\frac{mn}{d}$) is slightly better than that for {\tt adjoint} ($\frac{d +mn +1}{d}$).

\citet{kozak2021stochastic} analyze an algorithm very similar to {\tt RA-GD} with {\tt RA}={\tt pseudoinverse} and derive the same convergence rate \cite[Theorem 3]{kozak2021stochastic}. The difference between their algorithm and ours is that instead of using the pseudoinverse to obtain the projection of $\nabla f(X_k)$ on to a random subspace, they use an adjoint gradient estimator but with orthogonal sampling directions. More specifically, and in our notation, they sample $Z_1,\ldots, Z_d$ obeying Assumption~\ref{assumption:sampling_directions} and then use the QR factorization to obtain orthonormal $Q_1,\ldots, Q_d$ such that 
 \begin{equation*}
     \mathrm{Span}(Q_1,\ldots, Q_d) = \mathrm{Span}(Z_1,\ldots, Z_d).
 \end{equation*}
The convergence rate for {\tt RA-GD} with {\tt RA}={\tt adjoint} is analogous to those for gradient descent schemes using the RGE \eqref{eq:RGE}, see \citep{nesterov2017random,berahas2022theoretical}.

\begin{table*}[t]
    \centering
    \renewcommand{\arraystretch}{2} 
    \begin{tabular}{|c|c|c|c|}
    \hline
    \vspace{-3pt}
    {\tt RA:}& Type of guarantee & 
    Step Size & Number of Queries \\ \hline 
    {\tt adjoint}   & in expectation & $\displaystyle \alpha = \frac{d}{(d + mn +1) L}$ & $\displaystyle \frac{2 L (d + mn + 1)\Delta_f}{\varepsilon}$ \\ 
    {\tt pseudoinverse} & in expectation  & $\displaystyle \alpha = \frac{1}{L}$ & $\displaystyle \frac{2L mn \Delta_f }{\varepsilon}$ \\
     {\tt IHT} & with probab. $\displaystyle 1-e^{-c_1d}$ & $\displaystyle \alpha = \frac{1}{L}$  & $\displaystyle \frac{2c_0r(m+n)L\Delta_f}{\varepsilon\left(1 - \delta^2C_{\text{\tt IHT}}(m,r)\right)}$   \\ 
    \hline
    \end{tabular}
    \caption{Number of directional derivative evaluations (``queries'') required to find an $\varepsilon$-stationary point, assuming $d < mn$. Here $\Delta_f = f(X_0) - f_{\mathrm{low}}$.}
    \label{tab:rank_agnostic_convergence}
\end{table*} 

\subsection{The rank aware case}
\label{sec:rank_aware_gradient_recovery}
Next we consider the rank-aware case. We focus on {\tt RA-GD} with {\tt RA}={\tt IHT}.  The analysis of the other rank-aware gradient recovery methods discussed in Section~\ref{sec:Recovery_Algorithms} follows a similar path, and we do not pursue this here.

\begin{theorem}
\label{thm:IHT_gradient_recovery}
    Suppose Assumptions~\ref{assumption:L_smooth}, \ref{assumption:low-rank}, and \ref{assumption:sampling_directions} hold and let $d  = 2c_0r(m+n)$. Then $G_k$ produced by {\tt IHT} (given sufficiently many internal iterates) satisfies:
    \begin{equation*}
        \|G_k  - \nabla f(X_k)\|_\fro^2 \leq \left(8C\frac{m}{r} + 5\right)\delta^2\|\nabla f(X_k)\|_\fro^2 
    \end{equation*}
    with probability at least $1 - e^{-c_1d}$, where $c_0,c_1$, and $C$ are universal constants (i.e.\ they do not depend on $f$). 
\end{theorem}
The proof is in Appendix~\ref{AP:rankAwareConvergence} along with an elaboration on the condition of sufficiently many iterates. We note that in practice we use the step-size rule of \citet{tanner2013normalized}, while in theory we consider a step-size determined by the RIP constant of $\cS$. For notational convenience we define
\begin{equation*}
    C_{\text{\tt IHT}}(m,r) = 8C\frac{m}{r} + 5.
\end{equation*}

Combining Theorem~\ref{thm:IHT_gradient_recovery} with Lemma \ref{HighProbabilityConvergenceLemma} we obtain

\begin{theorem}
\label{thm:IHT_convergence}
 Suppose Assumptions~\ref{assumption:L_smooth}, \ref{assumption:low-rank}, and \ref{assumption:sampling_directions} hold and let $d  = c_0r(m+n)$. Suppose  $C_{\text{\tt IHT}}(m,r)< \delta^{-2}$. Then, after running $K$ iterations of Algorithm \ref{alg:RA-GD}, using IHT with step size $\alpha = 1/L$, we find

 \begin{equation*}
     \min_{0\leq k \leq K-1} \|\nabla  f(X_k)\|_\fro^2 \leq \frac{2L(f(X_0) - f_\text{low})}{(1-\delta^2C_{\text{\tt IHT}}(m,r)) K} 
 \end{equation*}
 with probability at least $1-e^{-c_1 d}$, where $c_0,c_1,$ and $C$ are the constants specified in Theorem \ref{thm:IHT_gradient_recovery}.
\end{theorem}
Rearranging Theorem~\ref{thm:IHT_convergence} yields the query complexity stated in line 3 of Table~\ref{tab:rank_agnostic_convergence}.

\subsection{Spectral Descent}

Each iteration of {\tt IHT} computes an SVD, in order to compute $\Projr$. Consequently, we may easily modify {\tt IHT} to return the SVD of $G_k$ at no extra cost. So, the projection of $G_k$ on to the set of semi-orthogonal matrices (see \eqref{eq:semi-orthogonal-matrices}) is computable at the cost of one $m\times r$ matrix-matrix multiply:
\begin{equation}
    \mathrm{Proj}_{\mathcal{O}_{m,n}}(G_k) = U_{G_k}V_{G_k}^{\top}.
\end{equation}
Thus instead of vanilla gradient descent one may wrap the {\tt IHT} gradient into {\em Spectral Descent}, see \eqref{eq:spectral_descent}.

\begin{algorithm}
\caption{{\tt IHT-SpecGD}.}
\label{alg:SpecGD}
\begin{algorithmic}[1]
\STATE \textbf{Input:} $X_0$, number of samples per iteration $d$
\STATE Sample $Z_1,\ldots, Z_d$ satisfying Assumption~\ref{assumption:sampling_directions}.
\FOR{ $k=0,\ldots, K-1$}
    \STATE Acquire $y_i = D_{Z_i}f(X_k)$ for $i=1,\ldots, d$.
    \STATE Acquire {\bf the SVD of the gradient estimate}: 
    \begin{equation*}
        U_k,\tilde{\Sigma}_k, V_k \text{ where } U_k\tilde{\Sigma}_kV_k^{\top} = G_k = {\tt IHT}(\vec{y}; Z_1,\ldots, Z_d).
    \end{equation*}
    \STATE Acquire step-size $\alpha_k$.
    \STATE Iterate: $X_{k+1} = X_k - \alpha_k U_kV_k^{\top}$
\ENDFOR
\STATE \textbf{Return:} $X_K$.
\end{algorithmic}
\end{algorithm}

We analyze the convergence of {\tt IHT-SpecGD} under a different geometry to that considered in the preceding sections. Specifically, we assume smoothness with respect to the {\em nuclear norm}, as captured by Assumption~\ref{assumption:L_smooth_nuclear}.

\begin{theorem}
    \label{thm:spec_gd_convergence}
    Suppose Assumptions~\ref{assumption:bounded_below}, \ref{assumption:L_smooth_nuclear}, \ref{assumption:low-rank}, and \ref{assumption:sampling_directions}  hold. Set $\alpha_k = 1/\sqrt{K}$ for all $k \geq 0$, and $K\geq 3$. Then, with probability at least $1 - e^{-c_1d}$, the iterations of Algorithm~\ref{alg:SpecGD} satisfy 
    \begin{equation}
        \min_{0\leq k \leq K-1}\|\nabla f(X_k)\|_\star \leq \frac{1}{1-2\sqrt{mC_{\text{\tt IHT}}(m,r)}\delta}\left(\frac{f(X_0) - f_{\mathrm{low}}}{\sqrt{K}} + \frac{L_{\star}}{2K^{3/2}}\right).
    \end{equation}
\end{theorem}

\begin{remark}
Ignoring the term proportional to $K^{-3/2}$, we obtain a similar query complexity to those displayed in Table~\ref{tab:rank_agnostic_convergence}. That is, {\tt IHT-SpecGD} finds an $\varepsilon$ stationary point, in the sense of \eqref{eq:eps_stationarity} but with the Frobenius norm replaced with the nuclear norm, after making on the order of 
\begin{equation*}
    \frac{c_0r(m+n)\Delta_f^2}{\varepsilon(1-2\sqrt{mC_{\text{\tt IHT}}(m,r)}\delta)^2}
\end{equation*}
queries. Note the dependence on $\Delta_f^2$, not $\Delta_f$.
\end{remark}

\section{Experiments} \label{sec:numerical_experiments}
To understand which approaches are worth pursuing further, we study their performance on two synthetic test functions. We define the $r$-singular-value-squared function as
\begin{equation}
  f_{r\sigma}: \mathbb{R}^{m\times m} \to \mathbb{R} \qquad   f_{r\sigma}(X) := \sum_{i=1}^r\sigma_i(X)^2
\end{equation}
which reduces to the strongly-convex and smooth function $\|X\|_\fro^2$ if $r=m$. 
Note $f_{r\sigma}$ has gradients of rank at most $r$ everywhere. The Ky Fan regression function is
\begin{equation}
   f_{mr}: \mathbb{R}^{m\times m} \to \mathbb{R} \qquad  f_{mr}(X) = \frac{1}{2}\|X - X^{\star}\|_\mathrm{KF}^2
\end{equation}
where $\|\cdot\|_{\mathrm{KF}}$ is the $r$ Ky Fan norm:
\begin{equation*}
    \|A\|_{\mathrm{KF}} = \sigma_1(A) + \ldots + \sigma_r(A)
\end{equation*}
 and $X^{\star}$ is a fixed matrix\footnote{We generate $X^{\star}$ randomly, but fix the seed across all replicates in our experiment so that $X^{\star}$ is always the same.}. This function also has gradients of rank at most $r$ everywhere, but in the case $r = m$ reduces to the convex but non-smooth function $\|\cdot\|_{\star}$.

 We consider the following suite of optimizers: the $d$-sample random gradient method (\eqref{eq:random_gradient_method}, equivalently: {\tt RA-GD} with {\tt RA}={\tt adjoint}); LOZO \cite{chenenhancing2025}; {\tt RA-GD} with {\tt RA} = {\tt AltMin}, {\tt BM}, {\tt IHT} and {\tt pseudoinverse}; and Algorithm~\ref{alg:SpecGD}. For each (optimizer, objective function) pair we run ten independent replicates given a query budget of $100,000$. In each replicate the only randomness is in the choice of sampling directions. That is, $X_0$ (and $X^{\star}$ in the case of $f_{\rm mr}$) is fixed. For both $f_{r\sigma}$ and $f_{\rm mr}$ we set $m=30$ and $r=3$.
 
 We want to compare optimizers given the best possible step-sizes $\{\alpha_k\}_{k=0}^{K-1}$ and per-iteration sample budget $d$. We do so by (i) tuning $d$, the number of samples per iteration, for every method and every objective function; and (ii) selecting $\alpha_k$, the step-size in the $k$-th iterate, by line search. We {\em do not} count the function evaluations made during this line search towards the query budget. We set $\eta^t$, the inner step-size of {\tt BM-GD}, by a second line search. Additional implementation details can be found in Appendix~\ref{app:Implementation_Details}.

\begin{figure}[htb]
\centering
\includegraphics[width=0.48\textwidth]{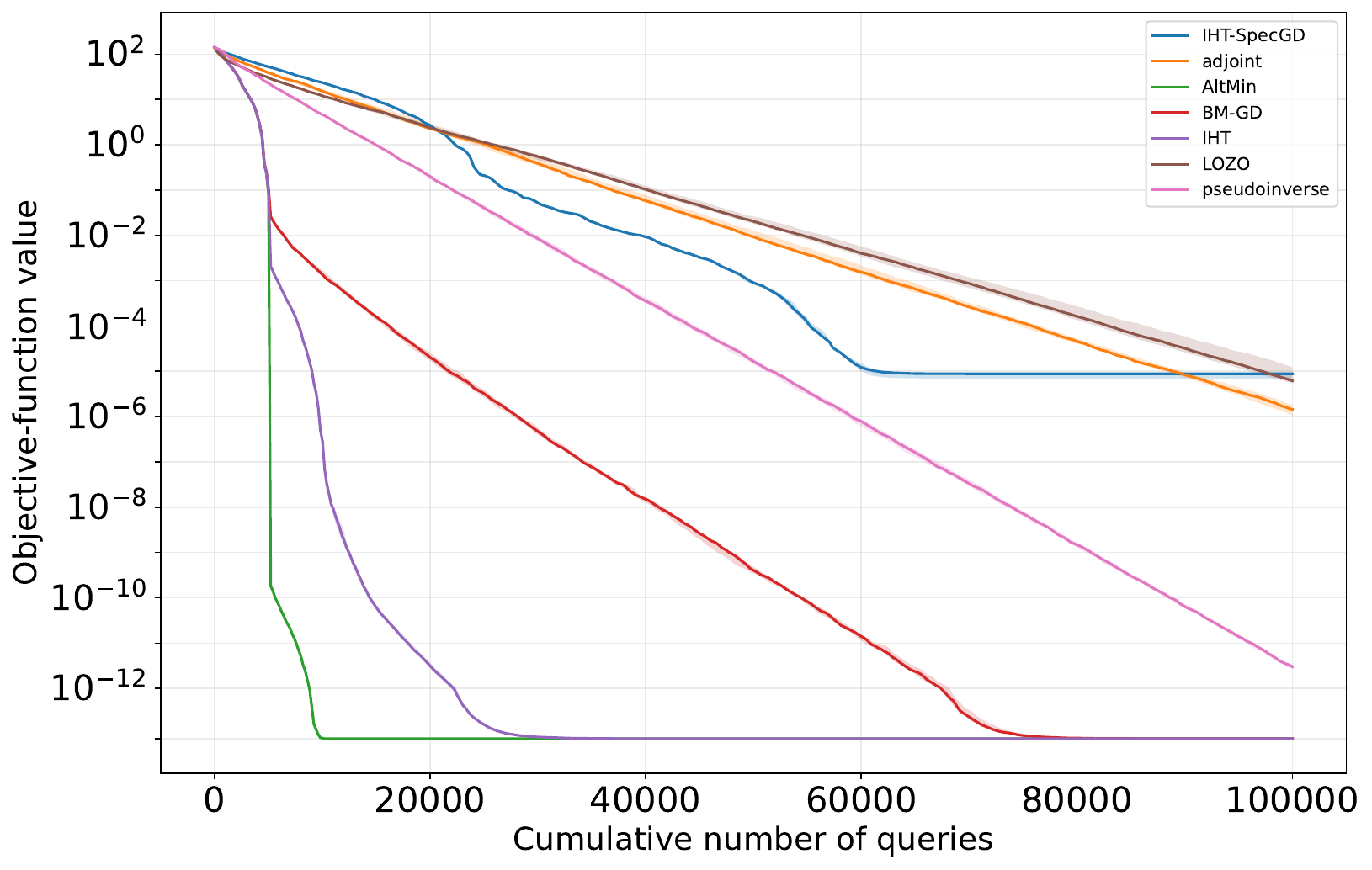}
\hfill
\includegraphics[width=0.48\textwidth]{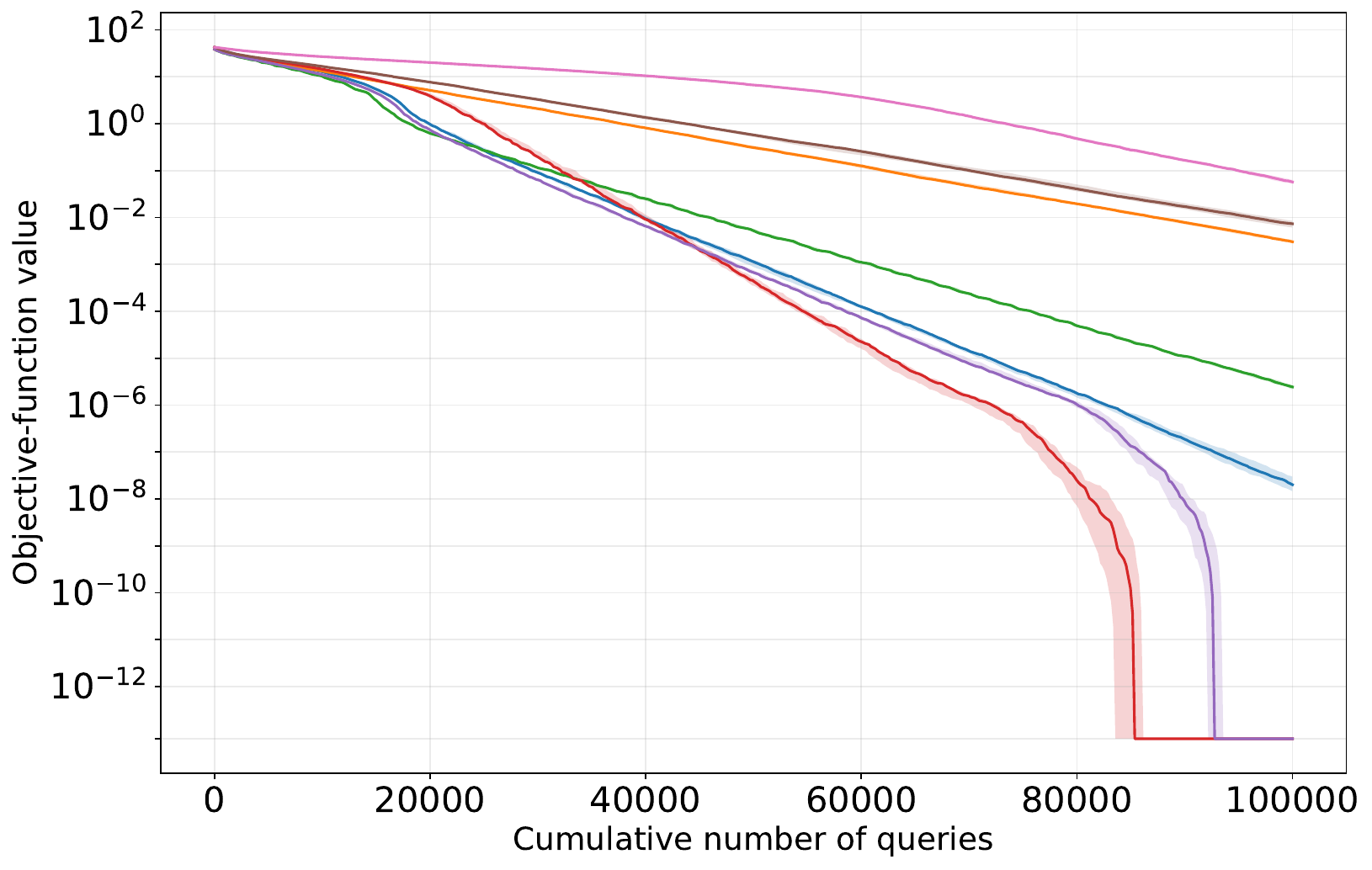}
\caption{Comparing a suite of optimizers on $f_{r\sigma}$ (left) and $f_{\rm mr}$ (right). Shading indicates the 25th--75th percentiles. Solid line indicates the median.}
\label{fig:experimental_results}
\end{figure}

Our results (Figure~\ref{fig:experimental_results}) show: (i) $f_{\rm mr}$ is uniformly more challenging than $f_{r\sigma}$. This is to be expected, as $f_{r\sigma}$ corresponds to a strongly convex function as $r \to m$, whereas $f_{\rm mr}$ becomes merely convex; (ii) with the occasional exception of {\tt IHT-SpecGD}, rank-aware optimizers handily outperform rank-agnostic optimizers, including LOZO. 

\section{Discussion and Limitations}

It appears that rank-aware optimizers for matrix functions warrant further study. However, in order to apply this algorithm to the LLM scale problems discussed in Sections~\ref{sec:introduction} and \ref{sec:low_rank_gradients}, several limitations must be overcome. First, Algorithms~\ref{alg:RA-GD} and \ref{alg:SpecGD} must be adapted to functions of sets of matrices, see Remark~\ref{remark:sets_matrices}. Second, the cost of storing the sampling directions $Z_1,\ldots, Z_d$ must be addressed. In principle this could be done using properties of pseudorandom number generators (pRNG), as proposed by \citet{malladi2023fine}. Specifically, instead of storing $Z_i$ one could just store the seed of the pRNG used, and re-generate $Z_i$ from scratch when needed. In practice, getting this to interact with sophisticated recovery algorithms such as {\tt IHT} appears challenging. Finally, standard first-order optimization tricks such as momentum and gradient-clipping ought to be incorporated.

The notion of gradient recovery, as presented in Section~\ref{sec:Recovery_Algorithms}, could also be explored further. Specifically, we have limited ourselves to a class of algorithms which choose the sampling directions $Z_1, \dots, Z_d$ before using them to estimate the gradient. However, {\em adaptive} sampling methods may well perform better.

Finally, we highlight the potential of {\tt IHT-SpecGD}. While it is not the most performant method in our simple experiments, given the success of {\tt Muon} and relatives on large-scale, first-order matrix optimization problems we feel that it should not be ruled out. 

\section{Conclusion}
We consider the problem of gradient free optimization for matrix functions. We propose a notion of low intrinsic dimensionality adapted to this setting, namely approximately low-rank gradients, and provide empirical evidence that this property occurs naturally when fine-tuning LLMs. When a matrix function satisfies this property techniques from the low-rank matrix sensing literature may be employed to provide improved gradient estimators. We have shown---in theory and in practice---that simple optimization schemes using such {\em rank-aware} gradient estimators are more efficient than rank-agnostic ones when gradients are approximately low rank. Future work will focus on extending the optimization algorithms proposed here to large-scale problems of interest, specifically LLM fine-tuning.


\section*{Statement on AI usage}
The authors used AI tools, particularly Claude Opus 4.8 and Claude Fable, for surfacing references, brainstorming several proofs (particularly Lemma 2), and refining the accompanying codebase. The authors reviewed all AI produced content and take full responsibility for the content of this work.

\section*{Funding}
SB and DM were partially supported by the National Science Foundation under Award No. 2608659. Any opinions, findings and conclusions or recommendations expressed in this material are those of the author(s) and do not necessarily reflect the views of the National Science
Foundation

\bibliography{bibliography}
\bibliographystyle{tmlr}

\appendix

\section{Convergence Lemmas}
We provide two variants of the standard convergence-to-stationarity result for gradient descent, adapted to the two settings in which we analyze the convergence of {\tt RA-GD}. 
Lemma~\ref{lemma:SGD-conv} can be deduced from \cite{bottou2018optimization} (see their Thm.~4.8 and the discussion about assumptions at Eq.~4.9), though we give a proof because it introduces inequalities that will be useful in our later analysis.  Note the improved convergence rate of $\mathcal{O}(1/K)$ compared to the standard rate of $\mathcal{O}(1/\sqrt{K})$ found for SGD methods based on subsampling (cf.~\cite{garrigos2023handbook,patel2022global}).

\begin{lemma}\label{lemma:SGD-conv}
    Suppose that Assumptions~\ref{assumption:bounded_below} and \ref{assumption:L_smooth} hold, and let $G ( X )$ be a unbiased gradient estimator for $\nabla f ( X ),$ that is, that $\EE [ G ( X ) | X ] = \nabla f ( X )$ for any $\RR^{m \times n}$-valued random matrix $X.$ If there exists $M > 0$ such that $G$ satisfies the bound 
    $$
    \EE [ \|G ( X )\|_\fro^2 | X] \leq M \| \nabla f ( X ) \|_\fro^2,
    $$  
    Algorithm~\ref{alg:RA-GD} with $G_k = G ( X_k )$ and $\alpha = \frac 1{ML}$ yields iterates satisfying
$$
\min_{0 \leq k \leq K-1} \EE [ \| \nabla f ( X_k ) \|_\fro^2 ] \leq
\frac{2LM (f(X_0) - f_{\rm low})}K.
$$
\end{lemma}

\begin{proof}[Proof of Lemma \ref{lemma:SGD-conv}]
For fixed $k,$ $L$-smoothness gives
\begin{align*}
f(X_{k+1}) 
- f(X_k) 
&\leq
\langle \nabla f (X_k), X_{k+1} - X_k \rangle_\fro + \frac L2 \| X_{k+1} - X_k \|_\fro^2  = -  \alpha 
\big (
\langle \nabla f (X_k), G_k \rangle_\fro 
- \frac{\alpha L}2 \| G_k \|_\fro^2
\big ).
\end{align*}
Taking a conditional expectation gives
\begin{align*}
\EE [ f(X_{k+1}) | X_k ] - f (X_k) 
& =
\EE [ f(X_{k+1} ) - f (X_k)| X_k ] \\
& \leq 
- \alpha \big ( \| \nabla f ( X_k ) \|_\fro^2 -
\frac{\alpha L}2 
\EE [\| G_k \|_\fro^2 | X_k]  \big )
\\ & \leq 
- \alpha \big ( \| \nabla f ( X_k ) \|_\fro^2 -
\frac{\alpha L}2  
M \| \nabla f ( X_k ) \|_\fro^2 \big ) \\
& = 
- \alpha \bigg ( 
1 - \frac{L M \alpha}{2}
\bigg ) \| \nabla f ( X_k ) \|_\fro^2  = 
- 
\frac 1{2 L M} 
\| \nabla f ( X_k ) \|_\fro^2.
\end{align*}

Rearranging and taking an expectation reveals an inequality which telescopes: 
\begin{align*}
& \EE \| \nabla f ( X_k ) \|_\fro^2 \leq 
2 L M ( \EE f ( X_{k+1} ) - \EE f ( X_k ) ) \\
\implies &
\frac 1K \sum_{k=0}^{K-1} \EE \| \nabla f ( X_k ) \|^2_\fro \leq 
\frac{2 L M (f(X_0) - f(X_K))}{K} \\
\implies & \min_{0 \leq k \leq K-1} \EE \| \nabla f ( X_k ) \|_\fro^2 \leq \frac{2 L M (f(X_0) - f_{\rm low}}{K} 
\end{align*}
as the minimum is always less than or equal to the average, and also $f(X_K) \geq f_{\mathrm{low}}$.
\end{proof}

Next we consider the variant where bounds on the gradient estimator hold with high probability.
Such bounds based on relative error in the gradient have been known since at least \citet{polyak1987introduction} (see Thm.~2 in Ch.~4), and been studied more recently in the zeroth-order community. In the deterministic case, the best known rates for the convex case, proven using computer-aided performance estimation techniques, are still $\mathcal{O}(1/K)$ albeit with better constants and looser conditions on the stepsizes~\citep[Thm.~3.1]{VernimmenWorstCase2026}. Recent work in the stochastic case~\citep{hallak2025study} has more refined bounds that are not needed for our analysis since we can avoid the union bound as discussed in Remark~\ref{rmk:no-union-bound-needed}.

 \begin{lemma}\label{HighProbabilityConvergenceLemma}
     Suppose Assumptions \ref{assumption:bounded_below} and \ref{assumption:L_smooth} 
     hold and that the gradient estimator $G(X)$ acquired in line 10 of Algorithm~\ref{alg:RA-GD} satisfies
     \begin{equation}
         \|G_k - \nabla f(X_k)\|_\fro^2 \leq c_2\|\nabla f(X_k)\|_\fro^2 \quad \text{ for } 1\leq k \leq K
         \label{eq:gradient_error_bound}
     \end{equation}
     with probability at least $1 - e^{-c_1d}$. If $\alpha = 1/L$ and $c_2 <1 $ then
     \begin{equation*}
         \min_{0\leq k \leq K-1}\|\nabla f(X_k)\|_\fro^2 \leq \frac{2L\left(f(X_0) - f_{\mathrm{low}}\right)}{(1-c_2)K} 
     \end{equation*}
     with probability at least $1 - e^{-c_1d}$. 
 \end{lemma}

 \begin{remark}\label{rmk:no-union-bound-needed}
     Note that we are assuming that with probability $1-e^{-c_1d}$ {\em all} $G_k$ are `good'. That is, they satisfy the error bound \eqref{eq:gradient_error_bound}. A much weaker statement would be that with probability $1-e^{-c_1d}$ a particular $G_k$ is `good'. This type of strong guarantee is possible when using {\tt IHT} (or {\tt AltMin}, or {\tt BMGD}) for gradient recovery, as success is guaranteed as long as the sensing operator $\mathcal{Z}$ satisfies the RIP. As we only sample $\mathcal{Z}$ once, it either has the RIP for all iterates or it does not. The former holds with probability $1-e^{-c_1d}$. 
 \end{remark}

 \begin{proof}[Proof of Lemma \ref{HighProbabilityConvergenceLemma}]
 By Lipschitz smoothness (equivalently, the standard descent lemma), for any $1\leq k \leq K$,
     \begin{align*}
         f(X_{k+1}) \leq f(X_k) - \frac{1}{L}\langle \nabla f(X_k),G_k\rangle_\fro + \frac{1}{2L}\|G_k\|_\fro^2.
     \end{align*}
    Combining this with the identity
     \begin{equation*}
         \frac{1}{2}\|G_k\|_\fro^2 - \langle \nabla f(X_k),G_k\rangle_\fro  = \frac{1}{2}\|G_k - \nabla f(X_k)\|_\fro^2 - \frac{1}{2}\|\nabla f(X_k)\|_\fro^2
     \end{equation*}
     and the gradient error bound \eqref{eq:gradient_error_bound},
     \begin{align*}
          f(X_{k+1}) &\leq f(X_k) - \frac{1}{2L}\|\nabla f(X_k)\|_{\rm }^2 + \frac{1}{2L}\|G_k - \nabla f(X_k)\|_\fro^2 \leq f(X_k) - \frac{1}{2L}\|\nabla f(X_k)\|_\fro^2 + \frac{c_2}{2L}\|\nabla f(X_k)\|_\fro^2  
     \end{align*}
     with probability at least $1 - e^{-c_1d}$. Rearranging again yields an inequality that telescopes:
     \begin{align*}
        &  (1 - c_2)\|\nabla f(X_k)\|_\fro^2 \leq 2L\left(f(X_k) - f(X_{k+1})\right) \\ 
         \implies & \frac{1}{K} \sum_{k=0}^{K-1} \|\nabla f(X_k)\|_\fro^2\leq \frac{2L\left(f(X_0) - f(X_K)\right)}{(1-c_2)K}  \\
         \implies & \min_{0\leq k \leq K-1}\|\nabla f(X_k)\|_\fro^2 \leq \frac{2L\left(f(X_0) - f_{\mathrm{low}}\right)}{(1-c_2)K} 
     \end{align*}  
\end{proof}

\section{Convergence in the rank-agnostic case.} \label{appendix:rank-agnostic}
\newcommand{\Z}{\mathbf{Z}}
We analyze the $\text{\tt RA} = {\tt pseudoinverse}$ and ${\tt RA} = {\tt adjoint}$ cases, presenting proofs for the results of Section~\ref{sec:rank_agnostic_convergence}, as well as several auxiliary lemmas. Both recovery algorithms do not exploit, or benefit from, low-rank structure in $\nabla f(X_k)$. Consequently both algorithms may be analyzed by simply vectorizing the problem. Let $\Z \in \R^{d\times mn}$ be the matrix corresponding to $\cS$. That is, $\Z$ has rows $\mathbf{z}_i^\top = \mathrm{vec}(Z_i)^\top$. 

\subsection{Intermediate lemmas} 

We start with the proof of Lemma~\ref{lemma:moment_bounds_for_pseudoinverse} used for the $\text{\tt RA} = {\tt pseudoinverse}$  case.
\begin{proof}[Proof of Lemma \ref{lemma:moment_bounds_for_pseudoinverse}]
     The matrix corresponding to $\cS^{\dagger}\cS$ is $\Z^{\dagger}\Z$, which we recognize as the projection onto the row space of $\Z$, and
     \begin{equation*}
         \mathrm{Row Space}(\Z) = \mathrm{Span}(\vec{z}_1,\ldots, \vec{z}_d) \cong 
         \mathrm{Span}(Z_1,\ldots, Z_d)
     \end{equation*}
     where $\cong$ denotes the natural isomorphism corresponding to vectorization. 

To get our other results, we first note that by the rotational invariance of the Gaussian distribution, for every orthogonal matrix $Q \in \RR^{mn},$ we have
$$
Q^\top \Z^\dagger \Z Q = 
(\Z Q)^\dagger \Z Q \overset{d}{=} 
\Z^\dagger \Z.
$$
It follows that $Q^\top \EE [\Z^\dagger \Z] Q = \EE [\Z^\dagger \Z]$, i.e., that $\EE [\Z^\dagger \Z]$ commutes with all orthogonal matrices,  and thus one can show via using Householder transformations that every vector is an eigenvector and hence $\EE [ \Z^\dagger \Z ] = a I$ for some $a \in \RR$, where $I$ is the $mn\times mn$ identity matrix. Noting that $\mathrm{Row Space}(\Z)$ almost surely has dimension $\min (d, mn)$, it follows that, as $\Z^\dagger \Z$ is a projection onto this space, $\trace \Z^\dagger \Z = \min ( d, mn )$ almost surely, and therefore
$$
a =  \frac1{mn}\trace \EE [\Z^\dagger \Z] =
\frac {\min(d,mn)}{mn}.
$$

Thus, $\EE [\Z^\dagger \Z] = \frac{\min (d, mn)}{mn} I,$ so $\EE [\cS^\dagger \cS ] =\frac{\min (d, mn)}{mn} {\rm Id},$ and
$$
\EE [ G_k | X_k ] = 
\EE [ \cS^\dagger \cS \nabla f ( X_k ) | X_k ] =
\frac {\min (d, mn)}{mn} \nabla f ( X_k ).
$$
Using the property that $(\cS^\dagger \cS)^2 = \cS^\dagger \cS,$ we have
\begin{equation*}
\EE [ \| G_k \|_\fro^2 | X_k ] = 
\EE [
\trace ( \nabla f ( X_k )^\top
(\cS^\dagger \cS)^2 \nabla f ( X_k ) )
| X_k ] = 
\trace ( \nabla f ( X_k )^\top
\EE [
\cS^\dagger \cS
| X_k ] 
\nabla f ( X_k ) )
\end{equation*}
\begin{equation*}
= \frac{\min (d, mn)}{mn} \| \nabla f ( X_k ) \|_\fro^2.
\end{equation*}
\end{proof}

Next is the proof for Lemma~\ref{lemma:moment_bounds_for_adjoint}, relevant for ${\tt RA} = {\tt adjoint}$, using the same notation $\Z$ as above, and let $\ZZ_{ij}$ denote the $(i,j)$ entry of $\Z$.
\begin{proof}[Proof of Lemma \ref{lemma:moment_bounds_for_adjoint}]
    It is convenient to write $\vec{g}_k = \mathrm{vec}(G_k) \in \mathbb{R}^{mn}$ and $\vec{v}_k = \mathrm{vec}(\nabla f(X_k))\in\mathbb{R}^{mn}$. 
    Since $G_k := \cS^{*}(\vec{y})$ for $\vec{y}=\cS(\nabla f(X_k))$, we can write $\vec{g}_k = \Z^{\top}\Z  \vec{v}_k$.
    We have the following identities:
    \begin{align}
        \|G_k\|_\fro^2 &= \|\vec{g}_k\|_\euc^2 \hspace{3pt} \text{ and } \hspace{3pt} \|\nabla f(X_k)\|_\fro^2 = \|\vec{v}_k\|_\euc^2  \\
        \|\vec{g}_k\|_\euc^2 &= \vec{g}_k^{\top}\vec{g}_k = 
        \vec{v}_k^\top 
        \Z^\top \Z \Z^\top \Z 
        \vec{v}_k. \label{eq:quadS}
    \end{align}
    
    Recall $\Z \in \mathbb{R}^{d\times mn}$ has i.i.d.\ entries sampled from $\mathcal{N}(0,1/d)$ by Assumption~\ref{assumption:sampling_directions}, hence $\mathbb{E}\left[\Z^{\top}\Z\right]=I$ and so 
    \begin{equation*}
        \mathbb{E}[\vec{g}_k] = \mathbb{E}\left[\Z^{\top}\Z\right]\vec{v}_k = \vec{v}_k.
    \end{equation*}

    Expand the $(i,j)$-th entry of $\Z^\top \Z \Z^\top \Z$ as a sum:
    \begin{equation}
       \mathbb{E}\left[\left[\Z^{\top}\Z\Z^{\top}\Z\right]_{ij}\right] = \mathbb{E}\left[\sum_{k=1}^{mn}\sum_{\ell = 1}^{d}\sum_{p=1}^{d}\ZZ_{\ell i}\ZZ_{\ell k} \ZZ_{pk}\ZZ_{pj}  \right]  = \sum_{k=1}^{mn}\sum_{\ell=1}^{d}\sum_{p=1}^{d}\mathbb{E}\left[\ZZ_{\ell i}\ZZ_{\ell k} \ZZ_{pk}\ZZ_{pj}\right] \label{eq:fourth_moments}.
    \end{equation}
    Recall that all $\ZZ_{i_1i_2}$ are sampled i.i.d.\ from $\mathcal{N}(0,1/d)$. Analyzing the summands in \eqref{eq:fourth_moments} requires bounding fourth moments of Gaussian variables. This is done in \citet[Lemma 1]{nesterov2017random} using the bound:
    \begin{equation}
        \mathbb{E}\left[\|\vec{u}\|_\euc^p\right] \leq (p + n)^{p/2} \quad \text{ for } \vec{v} \in \mathbb{R}^n \text{ with } v_i \sim \mathcal{N}(0,1).
    \end{equation}
    Instead, in order to get an exact equality, we use Isserlis' theorem \cite{isserlis1916certain},
    \begin{align*}
        \mathbb{E}\left[\ZZ_{\ell i}\ZZ_{\ell k} \ZZ_{pk}\ZZ_{pj}\right] = \mathbb{E}\left[\ZZ_{\ell i}\ZZ_{\ell k}\right]\mathbb{E}\left[\ZZ_{pk}\ZZ_{pj}\right] +  \mathbb{E}\left[\ZZ_{\ell i}\ZZ_{pk}\right]\mathbb{E}\left[\ZZ_{\ell k}\ZZ_{pj}\right] + \mathbb{E}\left[\ZZ_{\ell i}\ZZ_{pj}\right]\mathbb{E}\left[\ZZ_{\ell k}\ZZ_{pk}\right].
    \end{align*}
    Each second moment is zero unless the indices coincide, in which case it is $1/d$. We deal with each term in the sum individually:
    \begin{align*}
        \sum_{k=1}^{mn}\sum_{\ell=1}^{d}\sum_{p=1}^{d} \mathbb{E}\left[\ZZ_{\ell i}\ZZ_{\ell k}\right]\mathbb{E}\left[\ZZ_{pk}\ZZ_{pj}\right]  
        = \delta_{ij}\sum_{\ell=1}^{d}\sum_{p=1}^{d}\mathbb{E}\left[\ZZ_{\ell i}\ZZ_{\ell i}\right]\mathbb{E}\left[\ZZ_{pj}\ZZ_{pj}\right] 
        = \delta_{ij}\sum_{\ell=1}^{d}\sum_{p=1}^{d}\frac{1}{d^2} = \delta_{ij}
    \end{align*}
    where $\delta_{ij} = \left\{\begin{array}{cc} 1 & \text{ if } i=j \\ 0 & \text{otherwise}\end{array}\right.$ is the Kronecker delta. Similarly
    \begin{equation*}
       \sum_{k=1}^{mn}\sum_{\ell=1}^{d}\sum_{p=1}^{d} \mathbb{E}\left[\ZZ_{\ell i}\ZZ_{pk}\right]\mathbb{E}\left[\ZZ_{\ell k}\ZZ_{pj}\right] = \delta_{ij}\sum_{\ell=1}^d\sum_{p=1}^d \mathbb{E}\left[\ZZ_{\ell i}\ZZ_{pi}\right]\mathbb{E}\left[\ZZ_{\ell j}\ZZ_{pj}\right] = \delta_{ij}\sum_{\ell=1}^d\left(\mathbb{E}\left[\ZZ_{\ell i}\ZZ_{\ell i}\right]\right)^2 = \sum_{\ell=1}^d\frac{1}{d^2} = \frac{1}{d}\delta_{ij}
    \end{equation*}
    and
    \begin{equation*}
        \sum_{k=1}^{mn}\sum_{\ell=1}^{d}\sum_{p=1}^{d} \mathbb{E}\left[\ZZ_{\ell i}\ZZ_{pj}\right]\mathbb{E}\left[\ZZ_{\ell k}\ZZ_{pk}\right] = \delta_{ij}\sum_{k=1}^{mn}\sum_{\ell=1}^d\mathbb{E}\left[\ZZ_{\ell i}\ZZ_{\ell i}\right]\mathbb{E}\left[\ZZ_{\ell k}\ZZ_{\ell k}\right] = \delta_{ij}\sum_{k=1}^{mn}\sum_{\ell=1}^d\left(\frac{1}{d^2}\right) = \delta_{ij}\frac{mn}{d}.
    \end{equation*}
    Consequently
    \begin{equation*}
        \mathbb{E}\left[\left[\Z^{\top}\Z\Z^{\top}\Z\right]_{ij}\right] = \left(1 + \frac{1}{d} + \frac{mn}{d}\right)\delta_{ij}
        \iff 
        \mathbb{E}\left[\Z^{\top}\Z\Z^{\top}\Z\right] = \left(1 + \frac{mn+1}{d}\right)I.
    \end{equation*}
    Recalling \eqref{eq:quadS}, we compute:

\begin{equation*}
        \EE\|\vec{g}_k\|_\euc^2 = 
        \vec{v}_k^\top 
        ( \EE \Z^\top \Z \Z^\top \Z ) \vec{v}_k 
= \left(1 + \frac{mn+1}{d}\right) \vec{v}_k^\top \vec{v}_k 
 = 
        \left(1 + \frac{mn+1}{d}\right) 
        \| \nabla f ( X_k ) \|_\fro^2.
    \end{equation*}
\end{proof}

\subsection{Main results}

\begin{theorem}\label{thm:rankagnostic}
    Suppose Assumptions 1--2 and 4 hold. Then, Algorithm~\ref{alg:RA-GD} with {\tt RA} = {\tt adjoint} and 
    \begin{equation*}
        \alpha = \frac{1}{(1 + \frac{mn + 1}d ) L}
    \end{equation*}
    satisfies 
    \begin{equation*}
        \min_{0 \leq k \leq K-1} \EE \left[
\| \nabla f ( X_k ) \|_\fro^2 \right] \leq
\frac{2 L (1 + \frac{mn + 1}d ) (f (X_0) - f_{\rm low})}{K},
    \end{equation*}
and consequently, visits an $\varepsilon-$stationary (in expectation) point after making
$$
\frac{2 L (d + mn + 1) (f (X_0) - f_{\rm low})}{\epsilon}
$$
directional derivative evaluations.

Algorithm~\ref{alg:RA-GD} with {\tt RA} = {\tt pseudoinverse} and $\alpha = 1/L$ satisfies 
    \begin{equation*}
        \min_{0\leq k \leq K-1}\mathbb{E}\left[\left\|\nabla f(X_k)\right\|_\fro^2\right] \leq \frac{2L ( \frac{mn}{\min(d,mn)} ) \left(f(X_0) - f_{\mathrm{low}}\right)}{K},
    \end{equation*}
and consequently, visits an $\varepsilon-$stationary (in expectation) point after making
$$
\frac{2 L ( d\frac{mn}{\min(d,mn)} ) (f (X_0) - f_{\rm low})}{\epsilon}
$$
directional derivative evaluations.

\end{theorem}

\begin{proof}[Proof of Theorem \ref{thm:rankagnostic}]
The result for {\tt RA}={\tt adjoint} follows directly from applying the result from Lemma \ref{lemma:SGD-conv} to the bounds in Lemma \ref{lemma:moment_bounds_for_adjoint}.

The bounds in Lemma~\ref{lemma:moment_bounds_for_pseudoinverse} for {\tt RA}={\tt pseudoinverse} do not exactly match the hypotheses for Lemma \ref{lemma:SGD-conv}. However, if we define $G_k$ to be the pseudoinverse estimator applied to $f$ and $\tilde{G}_k = \frac{mn}{\min(d,mn)} G_k,$ we can see that Algorithm~\ref{alg:RA-GD} with step size $\alpha$ and gradient estimator $G_k$ is equivalent to Algorithm~\ref{alg:RA-GD} with step size $\tilde{\alpha} = \frac{\min(d,mn)}{mn} \alpha$ and gradient estimator $\tilde{G}_k.$ Furthermore, Lemma~\ref{lemma:moment_bounds_for_pseudoinverse} implies that
\begin{equation*}
  \EE [\tilde G_k | X_k] = \nabla f ( X_k ) \quad \text{ and } \quad \EE [ \| \tilde G_k \|_\fro^2 | X_k ] = \frac{mn}{\min(d,mn)} \| \nabla f ( X_k ) \|_\fro^2
\end{equation*}
whence we may apply Lemma~\ref{lemma:SGD-conv} with $M = \frac{mn}{\min(d,mn)}$ and $\tilde{\alpha} = \frac{\min(d,mn)}{mnL},  $ to get the desired convergence.
\end{proof}

\section{Convergence in the rank-aware case}  \label{AP:rankAwareConvergence}
We first establish conditions under which the RIP (Definition~\ref{def:rip}) holds, with high probability.  
\begin{lemma}
   Suppose Assumption \ref{assumption:sampling_directions} holds. Then there exist constants $c_0, c_1 > 0$, depending only on $\delta$, such that
     \begin{equation}
         \delta_s(\cS) \leq \delta \text{ with probability } 1 - e^{-c_1d}
         \label{eq:rip_prob_bound}
     \end{equation}
     whenever $d \geq c_0s(m + n)$.
     \label{thm:RIP_bounds}
\end{lemma}
See \citet[Theorem 2.3]{candes2011tight} for a proof. We restate and then prove Thm.~\ref{thm:IHT_gradient_recovery}.

\setcounter{theorem}{0} 
\begin{theorem}
    Suppose Assumptions~\ref{assumption:L_smooth}, \ref{assumption:low-rank}, and \ref{assumption:sampling_directions} hold. Choose $c_0$ in Lemma~\ref{thm:RIP_bounds} large enough such that $\delta_{2r}(\cS) < 1/3$ holds with probability $1 - e^{-c_1d}$. Suppose that $\hat{G} = G^t$ where $G^t$ is found using 
    \begin{equation*}
        t \geq \tilde{D}\log\left(\frac{\|\vec{y}\|_\euc}{\left(\tilde{C}\frac{m}{r} + 2\right)\delta^2\|\nabla f(X_k)\|_\fro^2}\right)
    \end{equation*}
    iterations of {\tt IHT} (see \eqref{eq:iterative_hard_thresholding}) with step-size $\eta = 1/(1+\delta_{2r}(\cS))$, where $\delta$ is the value occurring in the definition of $(\delta,r)$-low-rank (see Assumption~\ref{assumption:low-rank}). Then 
    \begin{equation*}
        \|\nabla f(X_k) - \hat{G}\|_\fro^2 \leq \left(8C\frac{m}{r} + 5\right)\delta^2\|\nabla f(X_k)\|_\fro^2
    \end{equation*}
    holds with probability $1 - e^{-c_1d}$, for universal constants $C,\tilde{C}$, and $\tilde{D}$. 
\end{theorem}

\begin{proof}
    Write $G^{\star} := \nabla f(X_k)$ as $G^{\star} = G^{\star}_r + G^{\star}_c$ where $G^{\star}_r = \Projr(G^{\star})$. Then 
    \begin{align*}
        \vec{y} &= \cS(G^{\star}) = \cS(G^{\star}_r) + \cS(G^{\star}_c)
    \end{align*}
    Letting $\|\cdot\|_{\star}$ denote the nuclear norm (i.e., the sum of the singular values),  
    \begin{align}
        \|\cS(G^{\star}_c)\|_\euc  &\leq \sqrt{1 + \delta_r(\cS)}\left(\|G^{\star}_c\|_\fro + \frac{1}{\sqrt{r}}\|G^{\star}_c\|_{\star}\right) \label{eq:from_goldfarb} \\
         &\leq
             \frac{2}{\sqrt{3}}\bigg(\delta\|\nabla f(X_k)\|_\fro + \frac{1}{\sqrt{r}}\left(\delta\|\nabla f(X_k)\|_\fro\sqrt{m-r}\right) \bigg) \label{eq:l_1_l_2_norm_relationship}\\
         & = \frac{2\delta}{\sqrt{3}}\left(1 + \sqrt{\frac{m-r}{r}}\right)\|\nabla f(X_k)\|_\fro \\ 
       & \leq \frac{4\delta}{\sqrt{3}}\sqrt{\frac{m}{2r}} \|\nabla f(X_k)\|_\fro\label{eq:Jensen}
    \end{align}
    Equation~\eqref{eq:from_goldfarb} follows from \cite[Prop. 2.8]{goldfarb2011convergence}. Equation~\eqref{eq:l_1_l_2_norm_relationship} follows from Assumption~\ref{assumption:low-rank} and 
    \begin{equation*}
    \|G^{\star}_c\|_{\star} = \sum_{i=r+1}^{m}\sigma_i(G^{\star})  \leq \sqrt{m-r}\sqrt{\sum_{i=r+1}^{m} (\sigma_i(G^{\star}))^2} \leq \left(\sqrt{m-r}\right)\left(\delta\|\nabla f(X_k)\|_\fro\right).
    \end{equation*}
    Equation~\eqref{eq:Jensen} follows from an application of Jensen's inequality:
    \begin{equation*}
        \frac{1}{2}\sqrt{1} + \frac{1}{2}\sqrt{\frac{m-r}{r}} \leq \sqrt{\frac{1 + \frac{m-r}{r}}{2}} = \sqrt{\frac{m}{2r}}
    \end{equation*}
Appealing to \citet[Theorem 1.2]{jain2010guaranteed}, (with $\vec{e} = \cS(G^{\star}_c)$, their $\vec{b}$ replaced by our $\vec{y}$, and setting their $\epsilon = \delta^2$), we obtain
\begin{equation}
        \|G^{\star}_r - \hat{G}\|^2_\fro \leq \frac{1}{1-\delta_{2r}(\cS)}\Bigg(C\frac{8\delta^2\|\nabla f(X_k)\|_\fro^2}{3}\frac{m}{r} + \delta^2\|\nabla f(X_k)\|_\fro^2\Bigg)
        \label{eq:IHT_error_bound}
\end{equation}
provided
\begin{equation}
    t \geq \frac{1}{\log(1/D)}\log \left(\frac{\|\vec{y}\|_\euc^2}{2(C\|\cS(G^{\star}_c)\|_\euc^2 + \delta^2\|\nabla f(X_k)\|_\fro^2)}\right).
    \label{eq:IHT_iterate_bound}
\end{equation}
where $C,D$ are universal constants. We simplify the error bound \eqref{eq:IHT_error_bound} first:
\begin{align*}
\|G^{\star} - \hat{G}\|_\fro^2  &\leq 2\|G^{\star}_r - \hat{G}\|^2_\fro + 2\|G^{\star}_c \|^2_\fro \\
& \leq\frac{2}{1-\delta_{2r}(\cS)} \left(C\frac{8}{3}\frac{m}{r} + 1 \right)
    \delta^2\|\nabla f(X_k)\|_\fro^2  + 2\delta^2\|\nabla f(X_k)\|_\fro^2 \\
& \leq \left(8C\frac{m}{r} + 5\right)\delta^2\|\nabla f(X_k)\|_\fro^2
\end{align*}
Next we simplify the iterate bound \eqref{eq:IHT_iterate_bound}. Appealing to \eqref{eq:Jensen} and defining $\tilde{C}$ appropriately yields
\begin{equation*}
   2(C\|\cS(G^{\star}_c)\|_\euc^2 + \delta^2\|\nabla f(X_k)\|_\fro^2) \leq 2\left(C\frac{16\delta^2\|\nabla f(X_k)\|_\fro^2}{3}\frac{m}{2r} + \delta^2\|\nabla f(X_k)\|_\fro^2\right) 
\end{equation*}
\begin{equation*}
   = \left(\tilde{C}\frac{m}{r} + 2\right)\delta^2\|\nabla f(X_k)\|_\fro^2. 
\end{equation*}
Finally we define $\tilde{D} = \frac{1}{\log(1/D)}$. 
\end{proof}

\begin{remark}
    Notice that the required number of iterates scales like $-\log\left(\|\nabla f(X_k)\|\right)$, and consequently increases as $\|\nabla f(X_k)\| \to 0$. In practice, we fix the number of iterates equal to $20$, which should suffice until $\|\nabla f(X_k)\| \lesssim e^{-20}$.  
\end{remark}
Then using the above result, we are ready to state and prove our main convergence result in this setting:
\begin{theorem}
 Suppose Assumptions~\ref{assumption:L_smooth}, \ref{assumption:low-rank}, and \ref{assumption:sampling_directions} hold and let $d  = c_0r(m+n)$. After running $K$ iterations of Algorithm \ref{alg:RA-GD}, using IHT with step size $\alpha = 1/L$, we find

 \begin{equation*}
     \min_{0 \leq k \leq K-1} \|\nabla  f(X_k)\|_\fro^2 \leq \frac{2L(f(X_0) - f_\text{low})}{(1-\delta^2(8C\frac{m}{r}+5)) K} 
 \end{equation*}
 with probability at least $1-e^{-c_1 d}$, where $c_0,c_1,$ and $C$ are the constants specified in Theorem \ref{thm:IHT_gradient_recovery}.
\end{theorem}

\begin{proof}
By Theorem \ref{thm:IHT_gradient_recovery}, we know 
\begin{equation*}
        \|\nabla f(X_k) - \hat{G}\|_\fro^2 \leq \left(8C\frac{m}{r} + 5\right)\delta^2\|\nabla f(X_k)\|_\fro^2
    \end{equation*}
will hold for every iteration of Algorithm \ref{alg:RA-GD} with probability $1-e^{-c_1 d}$. Thus, applying Lemma \ref{HighProbabilityConvergenceLemma} with
\begin{equation*}
   c_2 = \left(8C\frac{m}{r} + 5\right)\delta^2 := C_{\rm IHT}(m,r)\delta^2 
\end{equation*}
we have the desired result.
\end{proof}

\setcounter{theorem}{4} 

\subsection{Proof of Theorem~\ref{thm:spec_gd_convergence}}
\begin{proof}
    Setting $X^{+} = X_{k+1}$, $M = G_k$ and $\eta=\alpha_k$ in the descent lemma variant of \citet[Lemma~4]{he2025low} gives 
    \begin{equation}
        f(X_{k+1}) \leq  f(X_k) - \alpha_k\|\nabla f(X_k)\|_\star  + 2\alpha_k\|G_k - \nabla f(X_k) \|_{\star} + \frac{L_{\star}\alpha_k^2}{2}.
    \label{eq:lemma_from_he2025}
    \end{equation}
    since our Assumptions~\ref{assumption:bounded_below} and \ref{assumption:L_smooth_nuclear} are equivalent to the assumptions of \citet[Lemma~4]{he2025low}. Then 
    \begin{align*}
    \|G_k-\nabla f(X_k) \|_\star &\le \sqrt{m}\|G_k-\nabla f(X_k) \|_\fro \leq \sqrt{mC_{\text{\tt IHT}}(m,r)} \delta\|\nabla f(X_k)\|_\fro \leq  \sqrt{mC_{\text{\tt IHT}}(m,r)}\delta\|\nabla f(X_k)\|_\star
    \end{align*}
    with probability $1 - e^{-c_1d}$ by Theorem~\ref{thm:IHT_gradient_recovery}, having used the equivalence of norms $\|A\|_{\fro} \leq \|A\|_{\star} \leq \sqrt{m}\|A\|_\fro$. Substituting this into \eqref{eq:lemma_from_he2025}, rearranging, and letting $\alpha_k = 1/\sqrt{K}$ gives 
    \begin{equation*}
        \begin{split}
            f(X_{k+1}) - f(X_k) \leq - \frac{1 - 2\sqrt{mC_{\text{\tt IHT}}}\delta}{\sqrt{K}}\|\nabla f(X_k)\|_\star + \frac{L_{\star}}{2K},
        \end{split}
    \end{equation*}
    for all $1\leq k \leq K$, with probability $1 - e^{-c_1d}$. Telescoping and using Assumption~\ref{assumption:bounded_below}:
    \begin{align*}
        & \frac{1 - 2\sqrt{mC_{\text{\tt IHT}}}\delta}{\sqrt{K}}\sum_{k=0}^{K-1}\|\nabla f(X_k)\|_{\star} \leq f(X_0) - f_{\mathrm{low}} + \frac{L_\star}{2K} \\
    \implies & \frac{1}{K}\sum_{k=0}^{K-1}\|\nabla f(X_k)\|_{\star} \leq \frac{1}{1-2\sqrt{mC_{\text{\tt IHT}}}\delta}\left(\frac{f(X_0) - f_{\mathrm{low}}}{\sqrt{K}} + \frac{L_{\star}}{2K^{3/2}}\right).  
    \end{align*}
    The claim now follows by once more using the fact that the minimum is less than the mean.
\end{proof}


\section{Implementation of Recovery Algorithms}
\label{app:Implementation_Details}
The precise implementation of all algorithms considered may be found in the accompanying codebase, see supplementary material. For each (optimizer, objective function) we selected the most performant value of $d$, the number of samples per iteration, from $\{8, 32, 64, 128, 256, 512\}$. The selected values are displayed in Table~\ref{tab:tuned-hyperparameters}. We set the number of inner iterations for {\tt IHT}, {\tt BM-GD}, and {\tt AltMin} to $20$. Finally, we tune the rank of the sampling matrices in LOZO over $\{4, 8, 16\}$. For $f_{r\sigma}$ a rank of $16$ was best, while for $f_{\rm mr}$ a rank of $4$ was best.

\begin{table}[h]
  \centering
  \caption{Tuned number of samples per iteration ($d$) for each algorithm.
  All recovery-based methods use 20 recovery iterations.}
  \label{tab:tuned-hyperparameters}
  \begin{tabular}{lcc}
    \toprule
    Algorithm & $d$ for $f_{\rm mr}$ & $d$ for $f_{r\sigma}$ \\
    \midrule
    {\tt IHT-SpecGD}    & 256 & 512 \\
    {\tt adjoint}       & 128 & 128 \\
    {\tt AltMin}        & 256 & 512 \\
    {\tt BM-GD}         & 256 & 512 \\
    {\tt IHT}           & 256 & 512 \\
    {\tt LOZO}          &  64 &  64 \\
    {\tt pseudoinverse} & 256 & 512 \\
    \bottomrule
  \end{tabular}
\end{table}

\section{Connection to derivative-free optimization}\label{ap:connection_to_dfo}
In DFO, one only has access to function evaluations, not directional derivatives. As is well known, function evaluations can be used to approximate directional derivatives. Here we bound the error in this approximation.
\begin{lemma}
\label{lemma:directional_derivatives}
    Suppose Assumption~\ref{assumption:L_smooth} holds. Then
    \begin{equation}
    y_i := \frac{f(X+hZ_i) - f(X)}{h} = D_{Z_i}f(X) + h \eta_i 
    \label{eq:matrix_measurement},
\end{equation} 
where $|\eta_i| \leq \frac{L}{2}\| Z_i \|_\fro^2.$
\end{lemma}

\begin{proof}[Proof of Lemma~\ref{lemma:directional_derivatives}]
    For any $i \in \{1,\ldots,d\}$ define
    \begin{equation*}
        \eta_i = \frac{f(X+hZ_i) - f(X) -
    h D_{Z_i}f(X)}{h^2}
    \end{equation*}
    so that $\frac{f(X+hZ_i) - f(X)}{h} =
    D_{Z_i}f(X) + h \eta_i$. To bound $\eta_i,$ consider $\phi:\R \to \R$ given by $\phi ( \alpha ) = f ( X + \alpha h Z_i ) - f ( X ).$ Applying the fundamental theorem of calculus and the chain rule, we find that
    \begin{equation*}
        f ( X + h Z_i ) - f ( X ) = \phi ( 1 ) - \phi ( 0 ) = \int_0^1 \phi' ( \alpha ) {\rm d} \alpha = \int_0^1 h D_{Z_i}f( X + \alpha h Z_i ){\rm d} \alpha.
    \end{equation*}
    Subtracting $hD_{Z_i}f ( X )$ from both sides, using the fact that
    \begin{equation*}
        D_{Z_i}f ( X ) = \langle Z_i,\nabla f(X)\rangle   
    \end{equation*}
    as $f$ is continuously differentiable, and taking absolute values:
    \begin{align*}
    | f ( X + h Z_i ) - f ( X ) -hD_{Z_i}f ( X ) | & \leq 
    | h | \int_0^1 | \langle Z_i, ( \nabla f ( X + \alpha h Z_i ) - \nabla f ( X ) \rangle_\fro | {\rm d} \alpha \\
    & \leq | h | \int_0^1 \| Z_i \|_\fro \| \nabla f ( X + \alpha h Z_i ) - \nabla f ( X ) \|_\fro {\rm d} \alpha \\
    & \leq | h | \int_0^1 \| Z_i \|_\fro ( L \| \alpha h Z_i \|_\fro ) {\rm d} \alpha \\
    & = |h|^2 \| Z_i \|_\fro^2 L \int_0^1 \alpha {\rm d} \alpha = |h|^2 \frac{L}{2} \| Z_i \|_\fro^2,
    \end{align*} 
    where we have applied the Cauchy-Schwarz inequality for the Frobenius inner product. Dividing by $| h |^2$ both sides, we find that $|\eta_i | \leq \frac{L}{2} \| Z_i \|_\fro^2.$ 
\end{proof}

When $Z_i$ is drawn from the Gaussian ensemble, $\| Z_i \|_\fro^2$ is easy to bound:

\begin{lemma}
\label{lemma:aux_bounds_eta}
    Suppose that each entry of each $Z_i$ is chosen i.i.d from $\mathcal{N}(0,1/d)$. Let $\boldsymbol{\eta} \in \mathbb{R}^d$ have entries $\eta_i$ as in Lemma~\ref{lemma:directional_derivatives}. Then
    \begin{equation*}
        |\eta_i| \leq \frac{Lmn}{d} \qquad \text{ with probability at least } 1 - 2 e^{-Cmn} 
    \end{equation*}
\end{lemma}
\begin{proof}
    Observe that 
    \begin{equation}
        |\eta_i| \leq \frac{L}{2}\|Z_i\|_\fro^2 = \frac{L}{2} \sum_{k,\ell}^{m,n} [Z_i]_{k,\ell}^2 = \frac{L}{2d}\sum_{k,\ell}^{m,n} \tilde{Z}_{i,k,\ell}^2
        \label{eq:fist_bound_on_eta}
    \end{equation}
    where $\tilde{Z}_{i,k,\ell} \overset{iid}{\sim} \mathcal{N}(0,1)$. Thus, 
    \begin{equation*}
        \mathbb{E}\left[\sum_{k,\ell}^{m,n} \tilde{Z}_{i,k,\ell}^2\right] = mn, 
    \end{equation*}
    while
    \begin{equation*}
      \mathbb{P}\left[\left|\sum_{k,\ell}^{m,n} \tilde{Z}_{i,k,\ell}^2 - mn \right| \geq mn \right] \leq 2\exp(-Cmn)
    \end{equation*}
    by standard concentration of measure (e.g. Bernstein's inequality \cite[Corollary 2.8.3]{vershynin2018high}). Consequently
    \begin{equation}
        \mathbb{P}\left[\sum_{k,\ell}^{m,n} \tilde{Z}_{i,k,\ell}^2 \leq 2mn  \right] \geq 1 - 2\exp(-Cmn).
    \end{equation}
Combining this with \eqref{eq:fist_bound_on_eta} yields the stated bound.
\end{proof}

\begin{remark}
In Lemma~\ref{lemma:directional_derivatives}, it at first appears one could take $h\to 0$ and make the error arbitrarily small, but it is well-known that floating point errors make this impractical. For forward finite differences, the best balance of truncation and floating point error is achieved with $h\sim \sqrt{\epsilon_\text{machine}}$ where $\epsilon_\text{machine}$ is the machine epsilon~\citep{GillMurraySaundersWright83}.
In more detail, using $\hat{\phantom{g}}$ to denote floating point approximations, 
if we assume $|\hat{f}(X+hZ_i)-f(X+hZ_i)|\le \epsilon$ and 
$|\hat{f}(X)-f(X)|\le \epsilon$, 
and define $\hat{y}_i := \frac{\hat{f}(X+hZ_i) - \hat{f}(X)}{h}$, then the total error, including this floating point error (neglecting downstream floating point error, which is typically of smaller effect) and the truncation error of Lemma~\ref{lemma:directional_derivatives} gives
\[
\left| \hat{y}_i - D_{Z_i}f(X) \right| \le  h \eta_i + 2\frac{\epsilon}{h}
\]
and this error bound is minimized by choosing $h=\sqrt{\frac{2\epsilon}{\eta_i}}$. If $f(X)$ is $\mathcal{O}(1)$ then $\epsilon \approx \epsilon_\text{machine}$ and $\left| \hat{y}_i - D_{Z_i}f(X) \right| \lesssim 2\sqrt{2 \eta_i \epsilon_\text{machine}}$.
\end{remark}

However, the presence of an error term complicates the analysis of several of our proposed methods. Consequently, we leave this setting for future work.

\end{document}